\documentclass[letterpaper,11pt]{amsart}
\usepackage[utf8]{inputenc}
\usepackage[english]{babel}
\usepackage{amsmath}
\usepackage{amsfonts}
\usepackage{amssymb}
\usepackage{amsthm}
\usepackage[all]{xy}

\author{Gangyong Lee and Mauricio Medina-B\'arcenas}
\title{Finite $\Sigma$-Rickart modules\footnote{This is a modified and extended version of the version submitted for publication.}}

\newtheorem{thm}{Theorem}[section]

\newtheorem{cor}[thm]{Corollary}
\newtheorem{lem}[thm]{Lemma}
\newtheorem{prop}[thm]{Proposition}

\theoremstyle{definition}
\newtheorem{defn}[thm]{Definition}

\newtheorem{exam}[thm]{Example}

\newtheorem{rem}[thm]{Remark}

\numberwithin{equation}{section}

\DeclareMathOperator{\Add}{Add}
\DeclareMathOperator{\add}{add}
\DeclareMathOperator{\Ker}{Ker}
\DeclareMathOperator{\Img}{Im}

\DeclareMathOperator{\Mat}{Mat}
\DeclareMathOperator{\Hom}{Hom}
\DeclareMathOperator{\End}{End}

\DeclareMathOperator{\Rad}{Rad}

\newcommand{\ess}{\leq^\text{ess}}
\newcommand{\dleq}{\leq^\oplus}

\begin{document}
\address{}
\address{Gangyong Lee, Department of Mathematics Education, Chungnam National University}
\address{Yuseong-gu Daejeon 34134, Republic of Korea}
\address{e-mail: lgy999$@$cnu.ac.kr}
\address{}
\address{Mauricio Medina-B\'arcenas, Facultad de Ciencias F\'isico-Matem\'aticas, Benem\'erita}
\address{Universidad Aut\'onoma de Puebla, Av. San Claudio y 18 Sur, Col.San Manuel,}
\address{Ciudad Universitaria, 72570, Puebla, M\'exico.}
\address{e-mail: mmedina$@$fcfm.buap.mx}

%\today


\begin{abstract} 
In this article, we study the notion of a finite $\Sigma$-Rickart module, as a module theoretic analogue of a right semi-hereditary ring.
A module $M$ is called  \emph{finite $\Sigma$-Rickart} if every finite direct sum of copies of $M$ is a Rickart module.
It is shown
that any direct summand and any direct sum of copies of a finite $\Sigma$-Rickart module are finite $\Sigma$-Rickart
modules. We also provide generalizations in a module theoretic setting of the most common
results of semi-hereditary rings. Also, we have a characterization of a finite $\Sigma$-Rickart module in terms of its endomorphism ring.
In addition, we introduce $M$-coherent modules and provide a characterization of finite $\Sigma$-Rickart modules in terms of $M$-coherent modules. At the end, we study when $\Sigma$-Rickart modules and finite $\Sigma$-Rickart modules coincide.
Examples which delineate the concepts and results are provided.
\end{abstract}

\maketitle


\vspace{0.5cm}
\section{Introduction}\label{intro}
After \emph{hereditary rings} were introduced by Kaplansky in the earliest 50's, many mathematicians studied \emph{semi-hereditary rings} as a natural generalization of hereditary rings. Recall that a ring $R$ is said to be \emph{right semi-hereditary} if every finitely generated right ideal of $R$ is projective. In \cite{smallex} L. Small gives an example of a ring $R$ which is right semi-hereditary but $R$ is not right hereditary. Following the research on hereditary rings, many characterizations of semi-hereditary rings also have been made. For instance, in \cite{homalg} is proved that a ring $R$ is right semi-hereditary if and only if every finitely generated submodule of a projective $R$-module is projective. Later, in \cite[Theorem 2]{megibbenabsolutely} is shown that a ring $R$ is right semi-hereditary if and only if factor modules of absolutely pure $R$-modules are absolutely pure (also, if and only if factor modules of injective $R$-modules are f-injective). Chase in \cite{c1} characterizes a right semi-hereditary ring as a right coherent ring whose right ideals are flat. Very close to right semi-hereditary rings are right Rickart rings. A ring $R$ is said to be \emph{right Rickart} if the right annihilator of any element of $R$ is generated by an idempotent.
 Small in \cite[Proposition]{s} proves that a ring $R$ is right semi-hereditary if and only if $\Mat_n(R)$ is a right Rickart ring for every positive integer $n$. In 2012 Lee, Rizvi, and Roman in \cite{lrr3} extend Small's result with the theory of Rickart modules. A right $R$-module $M$ is called \emph{Rickart} if $\Ker\varphi$ is a direct summand of $M$ for any $\varphi\in\End_R(M)$. They prove that a ring $R$ is right semi-hereditary if and only if $R^{(n)}$ is a Rickart module for every positive integer $n$ \cite[Theorem 3.6]{lrr3}. Inspired by the last result, in this paper we define {finite $\Sigma$-Rickart modules}.
 A right $R$-module $M$ is called \emph{finite $\Sigma$-Rickart} if $M^{(n)}$ is a Rickart module for every $n>0$. We present many properties of these modules extending those for right semi-hereditary rings. For a module $M$, we will focus on  finitely $M$-generated modules as a more general concept of finitely generated modules. We study a finite $\Sigma$-Rickart module $M$ using the class 
 %of some modules
, $\add(M)$, which is the analogue to that of finitely generated projective modules, $\add(R)$, in the case of the ring $R$. 
%First we prove that a module $M$ is finite $\Sigma$-Rickart if and only if every element in the class $\add(M)$ has $D_2$ condition (Theorem \ref{semihered2}). 
To get the module theoretic version of Megibben's result \cite[Theorem 2]{megibbenabsolutely} mentioned above, we introduce the class $\mathfrak{F}_M$: in the case of the ring $M=R$, $\mathfrak{F}_R$ is the class of absolutely pure modules. We will compare $\mathfrak{F}_M$ with the class $\mathfrak{E}_M$ which is introduced in \cite{lm1}. 
Note that $\mathfrak{E}_R$ is the class of injective modules.

After the introduction and some preliminary background, in Section 2, finite $\Sigma$-Rickart modules are defined, some examples are presented, and the general properties of these modules are studied. It is proved that direct summands and finite direct sums of copies of a finite $\Sigma$-Rickart module inherit the properties (Lemma \ref{summandher}). It is shown that  $M$ is finite $\Sigma$-Rickart if and only if every finitely $M$-generated submodule of an element in $\add(M)$ has $D_2$ condition (Theorem \ref{semihered2}). 
We introduce the class $\mathfrak{F}_M$ for a right $R$-module $M$ and characterize finite $\Sigma$-Rickart modules in terms of this new class (Theorem \ref{semifac}) which is a module theoretic version of \cite[Theorem 2]{megibbenabsolutely}. 
At the end of the section we provide a characterization of an endoregular module in terms of a finite $\Sigma$-Rickart module as well as a characterization of von Neumann regular rings as a corollary (Theorem \ref{semiendo} and Corollary \ref{shabsvn}, respectively). 

Our focus in Section 3 is on the endomorphism ring of a finite $\Sigma$-Rickart module. We introduce the concept of $M$-coherent modules and we link it under intrinsically projective modules. That is, when $M$ is intrinsically projective, two characterizations for an $M$-coherent module $M$ in terms of the intersection property of finitely $M$-generated submodules of $M$ (Theorem \ref{coh3}) and in terms of when $\End_R(M)$ is a right coherent ring (Theorem \ref{coh}) are provided.
 These results will help us to characterize a finite $\Sigma$-Rickart module in terms of its endomorphism ring. A module $M$ is finite $\Sigma$-Rickart if and only if $S=\End_R(M)$ is a right semi-hereditary ring and $_SM$ is flat if and only if $M$ is intrinsically projective, $M$-coherent and all right $S$-ideals are flat (Theorem \ref{endoringsh}). 

When a module $_SM$ is flat where $S=\End_R(M)$ in Section 4, we prove that  
  $M\in \mathfrak{E}_M$ if and only if $S$ is a right self-injective ring, and 
  $M\in\mathfrak{F}_M$ if and only if $S$ is a right f-injective ring (Corollary \ref{sremss12}). Also, for $M=\bigoplus_{i=1}^n H_i^{(\ell_i)}$, if $H_i$ is an indecomposable endoregular module and  $H_i$ is $H_j$-Rickart  for all $1\leq i,j\leq n$ then $\End_R(M)$ is a semiprimary PWD (Corollary \ref{corrickperf}). Therefore as an application, if $M$ is a finite $\Sigma$-Rickart module and $P$ is any simple module such that $\Hom_{R}(M, P)=0$ then $M^{(\ell)}\oplus P^{(n)}$ is a finite $\Sigma$-Rickart module for any $\ell, n>0$ (Proposition \ref{mandsimple}). At the end, we characterize those finite $\Sigma$-Rickart modules with semiprimary endomorphism ring (Proposition \ref{ricksp}). This allows us to determine when the concepts of $\Sigma$-Rickart and fintie $\Sigma$-Rickart coincide (Corollary \ref{srfsr}).
%if $M$ is a finite $\Sigma$-Rickart module and $P$ is a finitely $M$-generated submodule of an element in $\add(M)$ then $P$ is isomorphic to a direct sum of finitely $M$-generated submodules of $M$ (Proposition \ref{kaplfMg}). Also in this section, we prove that
\vspace{.2cm}

Throughout this paper, $R$ is an associative ring with unity and $M$ is a unitary
right $R$-module. For a right $R$-module $M$, $S=\text{End}_R(M)$ will
denote the endomorphism ring of $M$; thus $M$ can be viewed as a
left $S$- right $R$-bimodule. For $\varphi \in S$, $\text{Ker} \varphi$ and
$\text{Im} \varphi$ stand for the kernel and the image of $\varphi$,
respectively. 
The notations  $N\leq M$, $N \unlhd M$, $N \leq^\text{ess} M$, and $N\leq^\oplus M$ mean that $N$ is  a
submodule, a fully invariant submodule, an essential submodule, and a direct summand of $M$, respectively. We use $M^{(n)}$ to denote the direct sum of $n$ copies of $M$. 
By $\mathbb{Q}$, $\mathbb{Z}$, and $\mathbb{N}$ we denote the set of rational, integer, and natural numbers, respectively. For $1<n\in\mathbb{N}$, $\mathbb{Z}_n$ denotes the
$\mathbb{Z}$-module $\mathbb{Z}/n\mathbb{Z}$.
We also denote $\mathbf{r}_R(N)=\{r\in R\,|\, Nr=0\}$ and $\mathbf{l}_S(N)=\{\varphi\in S\,|\, \varphi N=0\}$  for $N\leq M$, and $\mathbf{r}_S(I)=\{\varphi\in S\,|\, \varphi I=0\}$ for $I_S\leq S$.
\vspace{.2cm}

In \cite{lrr} was introduced the concept of Rickart modules and were presented many properties of them.

\begin{defn}
A right $R$-module $M$ is called \emph{Rickart} if $\Ker\varphi$ is a direct summand of $M$ for every endomorphism $\varphi\in \End_R(M)$ \cite{lrr}. $M$ is called \emph{endoregular} if $\End_R(M)$ is a von Neumann regular ring \cite{lrr5}.
\end{defn}

Recall that a module $M$ is said to have \emph{$D_2$ condition} if $\forall N \leq M$ with $M/N\cong M'\leq^\oplus M$, we have $N\leq^\oplus M$. Note that any Rickart module and any projective module satisfies  $D_2$ condition. Dually, $M$ is said to have \emph{$C_2$ condition} if $\forall N \leq M$ with $N\cong M'\leq^\oplus M$, we have $N\leq^\oplus M$.

\begin{prop}\label{ricd2}
The following statements hold true for a right $R$-module $M$:
\begin{enumerate}
\item[(i)] \emph{(\cite[Proposition 2.11]{lrr})} $M$ is a Rickart module if and only if $M$ has $D_2$ condition and $\Img\varphi$ is isomorphic to a direct summand of $M$ for all $\varphi\in\End_R(M)$.
\item[(ii)] \emph{(\cite[Theorem 1.11]{lrr5})}
$M$ is an endoregular module if and only if $M$ is a Rickart module and $M$ has $C_2$ condition.
\item[(iii)]\emph{(\cite[Proposition 2.26]{lrr5})} $M$ is a projective Rickart module if and only if $\Img \varphi$ is projective for each $\varphi\in \End_R(M)$.
\end{enumerate}
\end{prop}

A module $M$ is said to be \emph{$N$-Rickart} (or \emph{relatively Rickart to $N$}) if $\Ker\rho\dleq M$ for every homomorphism $\rho\in\Hom_R(M, N)$ \cite{rr2}.
\begin{thm}[{\cite[Theorem 2.6]{lrr3}}]\label{ricric}
Let $M$ and $N$ be modules. Then $M$ is $N$-Rickart if and only if for any direct summand $M'\dleq M$ and any submodule $N'\leq N$, $M'$ is $N'$-Rickart.
\end{thm}

%The next lemma will be used along all the paper.
%
%\begin{lem}[{\cite[Lemma 1.2]{lm1}}]\label{epidirsummand}
%Let $\rho:M\to N$ be an epimorphism and $L\leq N$. If $\rho^{-1}(L)\dleq M$ then $L$ is a direct summand of $N$.
%\end{lem}

In some results we will assume that $M$ has some projectivity conditions in order to get deeper results. The next lemma will be useful. 

\begin{lem}[{\cite[18.2]{wisbauerfoundations}}]\label{projprop}
The following statements hold true for a right $R$-module $M$:
\begin{enumerate}
\item[(i)] Consider $0\to N'\to N\to N''\to 0$ as a short exact sequence. If $M$ is an $N$-projective module then $M$ is $N'$- and $N''$-projective.
\item[(ii)] If $M$ is $N_i$-projective for right $R$-modules $N_1,\dots, N_\ell$, then $M$ is $\bigoplus_{i=1}^{\ell}N_i$-projective.
\end{enumerate}
\end{lem}

Recall that a right $R$-module $M$  is called \emph{$\Sigma$-Rickart} if every direct sum of copies of $M$ is a Rickart module \cite{lm1}.  Also, $\Add(M)$ denotes  the class of all right $R$-modules $K$ such that $K$ is isomorphic to a direct summand of $M^{(\mathcal{I})}$ for some nonempty index set $\mathcal{I}$ (see \cite[Definition 2.9]{lm1}).

\begin{prop}[{\cite[Proposition 2.11]{lm1}}]\label{herproj}
Let $M$ be a right $R$-module such that $R\in\Add(M)$. Then $M$ is a $\Sigma$-Rickart module if and only if $M$ is a projective $R$-module and $R$ is a right hereditary ring.
\end{prop}

\begin{thm}[{\cite[Theorem 4.6]{lm1}}]\label{endoringh}
The following conditions are equivalent for a finitely generated module $M$:
\begin{enumerate}
\item[(a)] $M$ is a $\Sigma$-Rickart module;
\item[(b)] $S=\End_R(M)$ is a right hereditary ring and $_SM$ is flat.
\end{enumerate}
\end{thm}

\vspace{0.2cm}

%%%%%%%%%%%%%%%%%%%%%%%%%%%%%%%%%%%%%%%%%%%%%%%%%%%%%%%%%%%%%%%%%%%%%%%%%%%%%%%%%
%%%%%%%%%%%%%%%%%%%%%%%%%%%%%%%%%%%%%%%%%%%%%%%%%%%%%%%%%%%%%%%%%%%%%%%%%%%%%%%%%
%%%%%%%%%%%%%%%%%%%%%%%%%%%%%%%%%%%%%%%%%%%%%%%%%%%%%%%%%%%%%%%%%%%%%%%%%%%%%%%%%
%%%%%%%%%%%%%%%%%%%%%%%%%%%%%%%%%%%%%%%%%%%%%%%%%%%%%%%%%%%%%%%%%%%%%%%%%%%%%%%%%%%%%%%%%%%%%%%%%%%%%%%%%%%%%%%%%%%%%%%%%%%%%%%%%%%%%%%%%%%%%%%%%%%%%%%%%%%%%%%%%%
%%%%%%%%%%%%%%%%%%%%%%%%%%%%%%%%%%%%%%%%%%%%%%%%%%%%%%%%%%%%%%%%%%%%%%%%%%%%%%%%%%%%%%%%%%%%%%%%%%%%%%%%%%%%%%%%%%%%%%%%%%%%%%%%%%%%%%%%%%%%%%%%%%%%%%%%%%%%%%%%%%
%%%%%%%%%%%%%%%%%%%%%%%%%%%%%%%%%%%%%%%%%%%%%%%%%%%%%%%%%%%%%%%%%%%%%%%%%%%%%%%%%%%%%%%%%%%%%%%%%%%%%%%%%%%%%%%%%%%%%%%%%%%%%%%%%%%%%%%%%%%%%%%%%%%%%%%%%%%%%%%%%%
%%%%%%%%%%%%%%%%%%%%%%%%%%%%%%%%%%%%%%%%%%%%%%%%%%%%%%%%%%%%%%%%%%%%%%%%%%%%%%%%%%%%%%%%%%%%%%%%%%%%%%%%%%%%%%%%%%%%%%%%%%%%%%%%%%%%%%%%%%%%%%%%%%%%%%%%%%%%%%%%%%
%%%%%%%%%%%%%%%%%%%%%%%%%%%%%%%%%%%%%%%%%%%%%%%%%%%%%%%%%%%%%%%%%%%%%%%%%%%%%%%%%%%%%%%%%%%%%%%%%%%%%%%%%%%%%%%%%%%%%%%%%%%%%%%%%%%%%%%%%%%%%%%%%%%%%%%%%%%%%%%%%%
%%%%%%%%%%%%%%%%%%%%%%%%%%%%%%%%%%%%%%%%%%%%%%%%%%%%%%%%%%%%%%%%%%%%%%%%%%%%%%%%%%%%%%%%%%%%%%%%%%%%%%%%%%%%%%%%%%%%%%%%%%%%%%%%%%%%%%%%%%%%%%%%%%%%%%%%%%%%%%%%%%
%%%%%%%%%%%%%%%%%%%%%%%%%%%%%%%%%%%%%%%%%%%%%%%%%%%%%%%%%%%%%%%%%%%%%%%%%%%%%%%%%%%%%%%%%%%%%%%%%%%%%%%%%%%%%%%%%%%%%%%%%%%%%%%%%%%%%%%%%%%%%%%%%%%%%%%%%%%%%%%%%%

\section{Finite $\Sigma$-Rickart Modules}
In this section, after we introduce $\Sigma$-Rickart modules in 2020 \cite{lm1}, we present another natural generalized notion which is called finite $\Sigma$-Rickart modules and obtain some of its basic properties. Note that, since proofs of some results are similar to those in \cite{lm1}, we will omit or include proofs for the convenience of the reader.

\begin{defn}
A right $R$-module $M$ is called \emph{finite $\Sigma$-Rickart} if every finite direct sum of copies of $M$ is a Rickart module.
\end{defn}

\begin{exam}
(i) $R_R$ is a finite $\Sigma$-Rickart module iff $R$ is a right semi-hereditary ring. 

(ii) Any $\mathcal{K}$-nonsingular continuous module is  finite $\Sigma$-Rickart.

(iii) Every $\Sigma$-Rickart module and every endoregular module are finite $\Sigma$-Rickart.

(iv) Any submodule of $\mathbb{Q}_\mathbb{Z}$ is a finite $\Sigma$-Rickart module. For, let $N$ be any submodule of $\mathbb{Q}_\mathbb{Z}$ and $\varphi:N^{(n)}\to N^{(n)}$ be any endomorphism for any $0<n\in\mathbb{N}$. Then $\Img\varphi$ is a torsion-free group. Hence $\Ker\varphi$ is a pure subgroup of $N^{(n)}$ by \cite[Ch.V, 26(d)]{fl}. Therefore $\Ker\varphi\dleq N^{(n)}$ by \cite[Lemma 86.8]{fl2}. Thus $N$ is a finite $\Sigma$-Rickart module.

(v) Let $R$ be a Dedekind domain and $M$ be a direct sum of finitely generated torsion-free $R$-modules of rank one. Then every submodule of $M$ is a finite $\Sigma$-Rickart module (\cite[Theorems 3 and 4]{k}). 

(vi) Every finitely generated free (projective) module over a right semi-hereditary ring is a finite $\Sigma$-Rickart module.

(vii) When $M=\bigoplus_{i\in\mathcal{I}} M_i$ with $M_i \unlhd M$ for all $i\in\mathcal{I}$, $\bigoplus_{i\in\mathcal{I}} M_i$ is a finite $\Sigma$-Rickart module if and only if $M_i$ is a finite $\Sigma$-Rickart module for all $i\in\mathcal{I}$.
\end{exam}

We have the implications for right $R$-modules:
\begin{equation}\label{implications}
\begin{matrix}&&\textsc{$\Sigma$-Rickart}&&\\ &&\Downarrow&&\\ \textsc{Endoregular}&\Longrightarrow&\textsc{finite $\Sigma$-Rickart}&\Longrightarrow&\textsc{Rickart}\end{matrix}
\end{equation}
%\[\textsc{Endoregular }\Longrightarrow\textsc{finite $\Sigma$-Rickart }\Longrightarrow\textsc{Rickart}.\]

The next examples show that each converse of the above implications is not true, in general.

\begin{exam} (i) $\mathbb{Z}[x]_{\mathbb{Z}[x]}$ is Rickart but it is not finite $\Sigma$-Rickart.\\
(ii) The localization of integers at a prime $p$, $\mathbb{Z}_{(p)}=\{\frac{a}{b}\,|\,a, b\in\mathbb{Z}, p\nmid b\}$,  is a finite $\Sigma$-Rickart $\mathbb{Z}$-module which is not endoregular.\\
(iii) If a module has $C_2$ condition then the three concepts in the low part of (\ref{implications}) coincide by Proposition \ref{ricd2}(ii). \\
(iv) Consider the $\mathbb{Z}$-module $M=\mathbb{Q}\oplus\mathbb{Z}$. Since $\mathbb{Z}^{(n)}$ is a nonsingular extending module for any $n\in\mathbb{N}$ and $E(\mathbb{Z}^{(n)})=\mathbb{Q}^{(n)}$, from \cite[Theorem 2.16]{lrr3} $M^{(n)}$ is a Rickart module for any $n\in\mathbb{N}$. Thus, $M$ is a finite $\Sigma$-Rickart module. However, $M$ is not a $\Sigma$-Rickart module. For, assume that $M$ is a $\Sigma$-Rickart module. Since $\mathbb{Z}\leq^\oplus M$, by Proposition \ref{herproj} $M$ is a projective module, a contradiction. 
\end{exam}

\begin{lem}\label{summandher}
\begin{enumerate}
\item [(i)] Every direct summand of a finite $\Sigma$-Rickart module is finite $\Sigma$-Rickart.
\item [(ii)] Every finite direct sum of copies of a finite $\Sigma$-Rickart module is finite $\Sigma$-Rickart.
\end{enumerate}
\end{lem}

\begin{proof}
(i) Let $M$ be a finite $\Sigma$-Rickart module and $N$ be a direct summand of $M$. Then $N^{(n)}$ is a direct summand of $M^{(n)}$ for all $0<n\in\mathbb{N}$. Since $M^{(n)}$ is a Rickart module, so is $N^{(n)}$. Thus, $N$ is a finite $\Sigma$-Rickart module.

(ii) Let $M$ be a finite $\Sigma$-Rickart module. Consider $M^{(n)}$ a direct sum of copies of $M$ for any $n\in\mathbb{N}$. 
Then $(M^{(n)})^{(m)}=M^{(nm)}$ is a Rickart module for all $n, m\in\mathbb{N}$. Therefore $M^{(n)}$ is a finite $\Sigma$-Rickart module.
\end{proof}

\begin{defn}
Let $M$ be a right $R$-module. Denote by $\add(M)$ the class of all right $R$-modules $K$ such that $K$ is isomorphic to a direct summand of $M^{(n)}$ for some $0<n\in\mathbb{N}$.
Note that $\add(R)$ consists of all finitely generated projective right modules over a ring $R$. 
\end{defn}

\begin{rem}
If $M$ is a right $R$-module such that $R\in\add(M)$, then $M$ is a projective left $S$-module where $S=\End_R(M)$.
For, since $R$ is in $\add(M)$, $M^{(n)}\cong R\oplus N$ for some right $R$-module $N$ and some $n\in\mathbb{N}$. Applying the functor $\Hom_R(\_,M)$ we get
$S^{(n)}\cong\Hom_R(R,M)\oplus\Hom_R(N,M)\cong M\oplus\Hom_R(N,M)$
as left $S$-modules. Thus, $_SM$ is projective. 
In addition, for the case of $\Add(M)$, if $M_R$ is finitely generated such that $R\in\Add(M)$, then $_SM$ is projective.
\end{rem}
% Note that if $M$ is a finitely generated right $R$-module such that $R$ is in $\add(M)$, then $M$ is a projective right $R$-module and $M$ is a projective left $S$-module.

The next proposition generalizes Lemma \ref{summandher}(ii).

\begin{prop}\label{ricksub}
A module $M$ is finite $\Sigma$-Rickart if and only if every element in $\add(M)$ is a finite $\Sigma$-Rickart module.
\end{prop}
%\begin{proof}
%Let $K$ be in $\add(M)$. Then there exist $n>0$ and a  module $N$ such that $K\cong N\dleq M^{(n)}$. 
%From Lemma \ref{summandher}(i)(ii), 
%$M^{(n)}$  and $N$ are finite $\Sigma$-Rickart. Therefore $K$ is a finite $\Sigma$-Rickart module. The converse is clear.
%\end{proof}

 Recall that a module $N$ is said to be \emph{finitely} \emph{$M$-generated} if there exists an epimorphism  $M^{(n)}\to N$ for some $0<n\in\mathbb{N}$. 
\begin{lem}\label{relrick}
For a finite $\Sigma$-Rickart module $M$, the following statements hold true:
\begin{enumerate}
\item[(i)] $M^{(m)}$ is  $M^{(n)}$-Rickart for every $0<m,n\in\mathbb{N}$.
\item[(ii)] For given $K\in\add(M)$, the intersection of two finitely $M$-generated submodules of $K$ is finitely $M$-generated.
\item[(iii)] The intersection of two finitely $M$-generated submodules of $M$ is finitely $M$-generated.
\end{enumerate}
\end{lem}
\begin{proof} (i) It directly follows from Theorem \ref{ricric}. 
(ii)  The proof is similar to that of \cite[Lemma 2.13]{lm1}. (iii) It is the special case of (ii) (see also Theorems \ref{coh3} and \ref{endoringsh}).
\end{proof}

\begin{cor}[e.g., {\cite[Corollary 4.60]{l}}] For a right semi-hereditary ring $R$, the intersection of two finitely generated ideals of $R$ is finitely generated.
\end{cor}
\begin{proof} It directly follows from Lemma \ref{relrick}(iii) (see also Corollary \ref{iprick} and Lemma \ref{Chase}).
\end{proof}

\begin{thm}\label{2}
 The following conditions are equivalent for a module $M$:
\begin{enumerate}
\item[(a)] $M$ is a finite $\Sigma$-Rickart module;
\item[(b)] every $K\in\add(M)$ satisfies the following two statements: 
\begin{enumerate}
\item[(1)] any finitely $M$-generated submodule of $K$ is in $\add(M)$; and
\item[(2)] any epimorphism $N\to K$ with $N$ finitely $M$-generated splits.
\end{enumerate}
\end{enumerate}
\end{thm}

\begin{proof} The proof is similar to that of \cite[Theorem 2.12]{lm1}.
\end{proof}

%\begin{rem}\label{projcond2}WE HAVE TO FIND A RIGHT POSITION.\\
%It is clear that any projective right $R$-module satisfies the condition (2) in Theorem \ref{2}. In fact, the conditions (1) and (2) in Theorem \ref{2} came up from the proof of the classical theorem which states that  finitely generated submodules of projective modules over a right semi-hereditary ring are projective. See \cite[Theorem 2.29]{l}.
%\end{rem}

The following examples show that Conditions (b1) and (b2) of Theorem \ref{2} are independent. 

\begin{exam}\label{excond}
(i) Let $M_\mathbb{Z}=\mathbb{Z}_{p^\infty}$ for a prime $p\in\mathbb{Z}$ and let $K\in\add(M)$ be arbitrary. 
%Then $K\oplus K'\cong M^{(m)}$ for some $m>0$. We can see that $\mathbb{Z}_p^{(m)}=\Soc(M^{(m)})=\Soc(K)\oplus \Soc(K')$, hence $\Soc(K)$ and $\Soc(K')$ are finitely generated.
%Since $\mathbb{Z}$ is a noetherian ring and $K$, $K'$ are injective modules, $K$ and $K'$ decompose as finite sums of  indecomposable modules. Therefore, by the Krull-Remak-Schmidt Theorem (see \cite[Chap.V, Corollary 5.5]{stenstromrings}) $K\cong M^{(n)}$ for some $n>0$. 
Consider $P$ as a finitely $M$-generated submodule of $K$. Then there exists an epimorphism $\rho:M^{(n)}\to P$ for some $n>0$. Since $M$ is divisible, so is $M^{(n)}$. It is a fact that epimorphic images of divisible groups are divisible, hence $P$ is divisible. This implies that $P\dleq K$.
%By the characterization of divisible abelian groups \cite[Theorem 10.28]{rotmanintroduction}, $P\cong \mathbb{Z}_{p^\infty}^{(\ell)}$ for some $\ell>0$. 
Thus $P\in\add(M)$. Therefore $M=\mathbb{Z}_{p^\infty}$ satisfies  Theorem \ref{2}(b1). 

Now, consider the epimorphism $\varphi:M\to M$ given by $\varphi(a)=ap$. Since $\varphi$ is not a monomorphism and $M$ is uniform, $\varphi$ does not split. Thus $M$ does not satisfy Theorem \ref{2}(b2).  Note that $M$ is not finite $\Sigma$-Rickart because $M$ is not a Rickart $\mathbb{Z}$-module. 
\vspace{.1cm}

(ii) Consider the ring
\[R=\left\{\left(\begin{smallmatrix}
a & (x,y) \\
0 & a
\end{smallmatrix}\right)\bigm| a\in \mathbb{Z}_2, (x,y)\in\mathbb{Z}_2\oplus\mathbb{Z}_2\right\}\]
with the usual addition and multiplication of matrices. Then $R$ is a commutative local artinian ring with maximal ideal
$I=\left\{\left(\begin{smallmatrix}
0 & (x,y)\\
0 & 0
\end{smallmatrix}\right)\bigm| (x,y)\in\mathbb{Z}_2\oplus\mathbb{Z}_2\right\}.$\\
Let $M$ be a finitely generated free $R$-module. Then $M$ satisfies Theorem \ref{2}(b2) because every element in $\add(M)$ is projective. However, let $N$ be a simple submodule of $M$. 
Since $M$ is a free module, $N$ is finitely $M$-generated. Since $R$ is local, $\mathbf{r}_R(N)\ess R$. Thus $N$ is a singular simple right $R$-module. 
Hence $N$ is not projective, that is, $N$ is not in $\add(M)$. Therefore $M$ does not satisfy Theorem \ref{2}(b1). 
Note that $R$ is not a Rickart $R$-module because $\Ker\left(\begin{smallmatrix}
0 & (1,1)\\
0 & 0
\end{smallmatrix}\right)\leq^\text{ess}R_R$, hence
 $M$ is not finite $\Sigma$-Rickart. 
\end{exam}

\begin{thm}
If $M$ is a finite $\Sigma$-Rickart module then every finitely $M$-generated submodule $P$ of any element in $\add(M)$ is isomorphic to a direct sum of finitely $M$-generated submodules of $M$. 
\end{thm}

\begin{proof} The proof is similar to that of \cite[Theorem 2.14]{lm1}.
\end{proof}

%\begin{lem}
%Every quasi-projective module satisfies the condition \emph{(ii)} in Theorem \ref{2}. 
%\end{lem}
%
%\begin{proof}
%Since $M$ is quasi-projective, by \cite[18.1 and 18.2]{wisbauerfoundations} $M^{(n)}$ is quasi-projective for all $n>0$. Let $K\in\add(M)$ and $N$ be a finitely $M$-generated module of $K$ and $f:N\to K$ be an epimorphism. Then  there exists an epimorphism $\rho:M^{(m)}\to N$ for some $m>0$ and $K$ is isomorphic to a direct summand of $M^{(\ell)}$ for some $\ell>0$. Let $n=max\{m,\ell\}$. Since $M^{(n)}$ is quasi-projective, $K$ is $M^{(n)}$-projective. Hence there exists a homomorphism $g:K\to M^{(n)}$ such that $f\rho g=1_K$. Thus $f$ splits.
%\[\xymatrix{ & & K\ar@{--{>}}[dll]_g \ar[d]^{1_K} \\ M^{(n)}\ar[r]_\rho & N\ar[r]_f & K}\] 
%\end{proof}

%\begin{rem}
%If $M$ is finitely generated then every finitely $M$-generated module is finitely generated.
%\end{rem}
%Recall that a module $M$ is said to have \emph{$D_2$ condition} if $\forall N \leq M$ with $M/N\cong M'\leq^\oplus M$, we have $N\leq^\oplus M$.
\begin{thm}\label{semihered2}
The following conditions are equivalent for a module $M$:
\begin{enumerate}
\item[(a)] $M$ is a finite $\Sigma$-Rickart module;
\item[(b)] every finitely $M$-generated submodule of any element in $\add(M)$ has $D_2$ condition.
\end{enumerate}
\end{thm}
\begin{proof} The proof is similar to that of \cite[Theorem 2.17]{lm1}.
%(a)$\Rightarrow$(b) Let $N$ be a finitely $M$-generated submodule of an element in $\add(M)$. By Theorem \ref{2}, $N$ is in $\add(M)$. By Proposition \ref{ricksub} $N$ is a finite $\Sigma$-Rickart module, in particular $N$ is Rickart. Therefore $N$ has $D_2$ condition.
%
%(b)$\Rightarrow$(a) Let $\varphi:M^{(n)}\to M^{(n)}$ be any endomorphism. Then $\Img\varphi$ is finitely $M$-generated and $M^{(n)}\oplus\Img\varphi\leq M^{(n)}\oplus M^{(n)}$. By hypothesis $M^{(n)}\oplus\Img\varphi$ has $D_2$ condition. Since $(M^{(n)}\oplus\text{Im}\varphi)/(\text{Ker}\varphi\oplus\text{Im}\varphi)\cong\text{Im}\varphi\leq^\oplus M^{(n)}\oplus\Img\varphi$, $\text{Ker}\varphi\leq^\oplus M^{(n)}$ by $D_2$ condition. 
\end{proof}

\begin{cor}\label{5}
The following conditions are equivalent for a module $M$:
\begin{itemize}
\item[(a)] $M$ is a quasi-projective finite $\Sigma$-Rickart module;
\item[(b)] every finitely $M$-generated submodule of any element in $\add(M)$ is $M$-projective;
\item[(c)] every finitely $M$-generated submodule of any element in $\add(M)$ is quasi-projective.
\end{itemize} 
\end{cor}

\begin{proof}
(a)$\Rightarrow$(b) Let $K$ be a finitely $M$-generated submodule of an element in $\add(M)$. By Theorem \ref{2}, $K\in\add(M)$, that is, $K$ is isomorphic to a direct summand of $M^{(n)}$ for some $n>0$. Since $M$ is quasi-projective, $M^{(n)}$ is $M$-projective. Hence $K$ is $M$-projective. 

(b)$\Rightarrow$(c) Let $K$ be a finitely $M$-generated submodule of an element in $\add(M)$. Then there exists an epimorphism $M^{(n)}\to K$ for some $n>0$. Since $K$ is $M$-projective, $K$ is $K$-projective by Lemma \ref{projprop}. Therefore $K$ is quasi-projective. 
(c)$\Rightarrow $(a) It follows from Theorem \ref{semihered2}.
\end{proof}

%The following example shows that if the condition that $M$ is $M$-projective in the statement (d) of last proposition is removed then (d) may not imply (c).
%
%\begin{exam}
%The $\mathbb{Z}$-module $\mathbb{Q}_\mathbb{Z}$ satisfies (d) by Theorem \ref{semihered2} but $\mathbb{Q}_\mathbb{Z}$ is not quasi-projective. Thus (d)$\nRightarrow$(c). 
%\end{exam}

\begin{cor}
Let $R$ be a Dedekind domain which is a complete discrete valuation ring. Then every torsion-free module of finite rank is a quasi-projective finite $\Sigma$-Rickart module.
\end{cor}

\begin{proof}
Let $M$ be a torsion-free $R$-module of finite rank. Let $N$ be a finitely $M$-generated submodule of an element in $\add(M)$. Then $N$ is torsion-free and has finite rank. Then $N$ is quasi-projective by \cite[Theorem 5.8]{rangaswamyquasiprojectives}. From Corollary \ref{5}, $M$ is a quasi-projective finite $\Sigma$-Rickart module. 
\end{proof}

It is well known that a ring $R$ is right semi-hereditary if and only if every finitely generated submodule of a right projective module is projective (\cite[Proposition 6.2]{homalg}). In the next result, we give more characterizations for right semi-hereditary rings.

\begin{cor}\label{ringsemi}
The following conditions are equivalent for a ring $R$:
\begin{enumerate}
\item[(a)] $R$ is a right semi-hereditary ring;
\item[(b)] every finitely generated submodule of any projective right $R$-module is projective;
\item[(c)] every finitely generated submodule of any projective right $R$-module is $R$-projective;
\item[(d)] every finitely generated submodule of any projective right $R$-module is quasi-projective;
\item[(e)] every finitely generated submodule of any projective right module has $D_2$ condition. 
\end{enumerate}
\end{cor}

%\begin{cor}
%The following conditions are equivalent for a ring $R$:
%\begin{enumerate}
%\item[(a)] $R$ is a right semi-hereditary ring.
%\item[(b)] Every finitely generated submodule of a projective right module has $D_2$ condition. 
%\end{enumerate}
%\end{cor}

In \cite{lm1} $\Sigma$-Rickart modules were characterized using a class of modules called $\mathfrak{E}_M$. 
For a right $R$-module $M$, it is denoted by $\mathfrak{E}_M$ the class of all right $R$-modules $A$ such that
for any monomorphism $\alpha: N\rightarrow  M$ with $N$ an $M$-generated module  and for any homomorphism $\beta: N\rightarrow A$, there exists $\gamma:M\to A$ such that $\beta=\gamma\alpha$. For the analogue of the above class related to finite $\Sigma$-Rickart modules, we introduce the following.

\begin{defn}
Let $M$ be a right $R$-module. Denote by $\mathfrak{F}_M$ the class of all right $R$-modules $A$ such that for any monomorphism $\alpha: N\rightarrow M$ with $N$ a finitely $M$-generated module and for any homomorphism $\beta: N\rightarrow A$, there exists $\gamma:M\to A$ such that $\beta=\gamma\alpha$.
\end{defn}

%\begin{defn}
%Let $M$ be a right $R$-module. A right $R$-module $A$ is said to be \emph{f$_M$-injective} if for any finitely $M$-generated submodule $N$ of $M$ and for any homomorphism $\beta:N\to A$, there exists a homomorphism $\gamma:M\to A$ such that $\gamma|_N=\beta$. We denote by $\mathfrak{F}_M$ the class of f$_M$-injective modules in this paper.
%Note that in the case of $M=R_R$, $L$ is said to be \emph{f-injective} if $L$ is f$_R$-injective (see \cite{guptaf}).
%\end{defn}

%We present the analogue of that class for finite $\Sigma$-Rickart modules. 

%\begin{defn}
%Let $M$ be a right $R$-module. Denote by $\mathfrak{F}_M$ the class of all right $R$-modules $A$ such that for every diagram 
%\[\xymatrix{0\ar[r] & N\ar[r]^\alpha \ar[d]_\beta & M^{(n)}\ar@{--{>}}[dl]^\gamma \\ & A & }\]
%with $n>0$ and $N$ finitely $M$-generated, there exists $\gamma:M^{(n)}\to A$ such that $\beta=\gamma\alpha$.
%\end{defn}
For a right $R$-module $M$, a right $R$-module $A$ is said to be \emph{f$_M$-injective} if for any finitely $M$-generated submodule $N$ of $M$ and for any homomorphism $\beta:N\to A$, there exists a homomorphism $\gamma:M\to A$ such that $\gamma|_N=\beta$.  Note that in the case of $M=R_R$, $A$ is said to be \emph{f-injective} if $A$ is f$_R$-injective (see \cite{guptaf}). We can easily see that the every element in $\mathfrak{F}_M$ is exactly f$_M$-injective as the following.

\begin{prop}
For a right $R$-module $M$, a module $A$ is in $\mathfrak{F}_M$ iff $A$ is f$_M$-injective.
\end{prop}

\begin{prop}\label{injadd}
For a right $R$-module $M$, a module $A$ is in $\mathfrak{F}_M$ if and only if for any monomorphism $\alpha: N\rightarrow K$ with $N$ a finitely $M$-generated module and $K\in\add(M)$, and for any homomorphism $\beta: N\rightarrow A$, there exists $\gamma:K\to A$ such that $\beta=\gamma\alpha$.
\end{prop}

\begin{proof} The proof is similar to that of \cite[Proposition 3.2]{lm1}.
%The ``if" part is clear. For the ``only if" part, it is enough to prove that for a module $A\in\mathfrak{F}_M$ the following diagram
%\[\xymatrix{0\ar[r] & N\ar[r]^\alpha \ar[d]_\beta & M^{(n)}\ar@{--{>}}[dl]^\gamma &  \\ & A & }\]
%with $N$ a finitely $M$-generated module can be completed commutatively with a homomorphism $\gamma:M^{(n)}\to A$. By Zorn's Lemma we can assume $\beta:N\to A$ cannot be extended to a homomorphism $L\to A$ for any finitely $M$-generated submodule $L$ of $M^{(n)}$ which contains $N$ properly. We claim that $N=M^{(n)}$. Let $\eta_i:M\to M^{(n)}$ be the canonical inclusion and set $M_i=\eta_i(M)$ for all $1\leq i\leq n$. Then $M^{(n)}=\bigoplus_{i=1}^n M_i$. 
%If $N\neq M^{(n)}$ there exists some $i\in\mathbb{N}$ such that $M_i\nsubseteq N$ and $1\leq i\leq n$. By Lemma \ref{relrick}(ii), $M_i\cap N$ is finitely $M$-generated. 
%Since $A$ is in $\mathfrak{F}_M$, $\beta|_{M_i\cap N}$ can be extended to a homomorphism $\psi:M_i\to M$ such that $\psi|_{M_i\cap N}=\beta$. Note that $M_i+N$ is a finitely $M$-generated submodule of $M^{(n)}$. Define $\gamma:M_i+N\to M$ as $\gamma(x+y)=\psi(x)+\beta(y)$. If $x+y=0$ then 
%\[\gamma(x+y)=\psi(x)+\beta(y)=\beta(x)+\beta(y)=\beta(x+y)=0.\]
%Thus, $\gamma$ is well defined. Therefore $\beta$ can be extended to $M_i+N$, a contradiction. Hence $N=M^{(n)}$, proving the claim.
\end{proof}

\begin{rem}\label{frim}
(i) We have the following contentions,
\[\mathfrak{E}_R\subseteq\{\text{all}~ M\text{-injective modules}\}\subseteq\mathfrak{E}_M\subseteq \mathfrak{F}_M=\{\text{all} ~\text{f}_M\text{-injective modules}\}\]
where $\mathfrak{E}_R=\{\text{all}~\text{injective modules}\}$. Note that $\mathfrak{F}_R=\{\text{all}~\text{f-injective modules}\}$.%\text{~and~}\mathfrak{F}_R\subseteq \mathfrak{F}_M

(ii) If every submodule of $M$ is finitely $M$-generated then every module in $\mathfrak{F}_M$ is $M$-injective.
\end{rem}

\begin{prop}\label{xiprop}
The following statements hold true for  a right $R$-module $M$:
\begin{enumerate}
\item[(i)]\label{dirsumgamma} $\mathfrak{F}_M$ is closed under direct products.
\item[(ii)]\label{prodxi} $\mathfrak{F}_M$ is closed under direct summands.
\item[(iii)]\label{fmc2} If $M$ is in $\mathfrak{F}_M$ then $M$ has $C_2$ condition.
\item[(iv)]\label{mggamma} If every finitely $M$-generated submodule of $A$ is in $\mathfrak{F}_M$, then $A$ is in $\mathfrak{F}_M$.
\end{enumerate}
\end{prop}
\begin{proof} All proofs are similar to those of \cite[Proposition 3.6]{lm1}.
However, we give the proof of (iv) for the convenience of the reader.
(iv) Let $\alpha:N\to M$ be a monomorphism with $N$ finitely $M$-generated and let $\beta:N\to A$ be any homomorphism. Since $N$ is finitely $M$-generated, $\Img\beta$ is finitely $M$-generated. Because $\Img \beta \subseteq A$, by hypothesis there exists $\gamma:M\to\Img\beta$ such that $\beta(N)=\gamma\alpha(N)$. Therefore $A\in\mathfrak{F}_M$.
\end{proof}

%\begin{prop}\label{dirsumxi}
%Let $A$ be in $\mathfrak{F}_M$ and $A'\dleq A$. Then $A'$ is in $\mathfrak{F}_M$.
%\end{prop}
%
%\begin{prop}\label{prodxi}
%$\mathfrak{F}_M$ is closed under direct product for every module $M$.
%\end{prop}
%
%\begin{lem}\label{monosplxi}
%Let $A$ be a finitely $M$-generated in $\mathfrak{F}_M$, then every monomorphism $\alpha:A\to K$ splits, for $K$ be in $\add(M)$.
%\end{lem}
%
%\begin{prop}\label{fmgxi}
%If every finitely $M$-generated submodule of $A$ is in $\mathfrak{F}_M$. Then $A$ is in $\mathfrak{F}_M$.
%\end{prop}

\begin{cor}
If every finitely generated submodule of $M$ is f-injective then $M$ is also f-injective.
\end{cor}

\begin{prop}\label{endofm}
The following conditions are equivalent for a module $M$:
\begin{enumerate}
	\item[(a)] $M$ is an endoregular module;
	\item[(b)] $M$ has $D_2$ condition and $\mathfrak{F}_M$=\emph{Mod}-$R$.
\end{enumerate}
\end{prop}
\begin{proof}
(a)$\Rightarrow$(b) It is clear that $M$ has $D_2$ condition. Let $L$ be any right $R$-module and let $N$ a finitely $M$-generated submodule of $M$.
Then there exists an epimorphism $\rho:M^{(n)}\to N$ for some $n>0$. 
Also, let $\alpha:N\to M$ be any monomorphism and $\beta:N\to L$ be any homomorphism.
By \cite[Corollary 3.15]{lrr5}, $M^{(n)}$ is an endoregular module, and hence $\alpha\rho(M^{(n)})=\alpha(N)$ is a direct summand of $M$.
%Take $\gamma=\begin{cases}\beta\alpha^{-1}&\text{on~} \alpha(N)\\0&\text{on~} M\setminus\alpha(N)\end{cases}$.
Take $\gamma=\beta\alpha^{-1}\oplus 0$.
Then $\gamma: M\rightarrow L$ is a homomorphism such that $\gamma\alpha= \beta$.  Therefore $L$ is in $\mathfrak{F}_M$.

(b)$\Rightarrow$(a) Let $\varphi:M\to M$ be any endomorphism of $M$. Then $\Img\varphi$ is finitely $M$-generated. Since $\Img\varphi$ is in $\mathfrak{F}_M$, the canonical inclusion $j:\Img\varphi\to M$ splits, that is, $\Img\varphi$ is a direct summand of $M$. By the $D_2$ condition, we can infer that $\Ker\varphi$ is a direct summand of $M$. Thus, $M$ is an endoregular module.
\end{proof}
%Since $N\cong \alpha(N)\leq^\oplus M$, $N\leq^\oplus M$ by $C_2$ condition (see Proposition \ref{ricd2}(ii)).

%\begin{proof}
%Let $L$ be any right $R$-module and $\beta:N\to L$ be any homomorphism with $N$ a finitely $M$-generated submodule of $M$. Then there exists an epimorphism $\rho:M^{(n)}\to N$ for some $n>0$. Denote $\iota:N\to M$ the canonical inclusion. By \cite[Corollary 3.15]{lrr5}, $M^{(n)}$ is an endoregular module, hence $\Img \iota\rho=N$ is a direct summand of $M$. This implies that $\beta$ can be extended to $M$. Thus $L$ is f$_M$-injective.
%\end{proof}

%\begin{proof}
%(a)$\Leftrightarrow$(b)$\Leftrightarrow$(d) follow from Theorem \ref{semiendo}.
%(b)$\Leftrightarrow$(c) follow from Lemma \ref{meggiben}.
%\end{proof}
An epimorphism  $\mu:A\to B$ is called an $M$\emph{-pure epimorphism} if for any homomorphism $\beta:M\to B$, there exists $\gamma:M\to A$ such that $\mu\gamma=\beta$ \cite{wisbauerfoundations} (see also \cite[Proposition 3.10]{lm1}).

\begin{rem}\label{mpepisprop}
It is not difficult to see that:
\begin{enumerate}
\item[(i)] An epimorphism $\mu:A\to B$ is an $M$-pure epimorphism if and only if $\mu$ is a $K$-pure epimorphism for any $K$ in $\add(M)$. 
\item[(ii)] For a projective module $M$, every epimorphism is an $M$-pure epimorphism.
\item[(iii)] If $\mu:A\to B$ and $\nu:C\to D$ are $M$-pure epimorphisms, then $\mu\oplus\nu:A\oplus C\to B\oplus D$ is also an $M$-pure epimorphism.
\end{enumerate}
\end{rem}

\begin{lem}[{\cite[Lemma 3.11]{lm1}}]\label{epipure}
Let $M$ be $M^{(\mathcal{I})}$-projective for any \emph{(}resp., finite\emph{)} index set $\mathcal{I}$. If $A$ is an \emph{(}resp., finitely\emph{)} $M$-generated module then every epimorphism $\mu:A\to B$ is an $M$-pure epimorphism.
\end{lem}
%\begin{proof} 
%Let $\beta:M\to B$ be any homomorphism. Since $A$ is an (finitely) $M$-generated module, there exists an epimorphism $\rho:M^{(\mathcal{I})}\to A$ for some index set $\mathcal{I}$. Since $M$ is $M^{(\mathcal{I})}$-projective for every (finite) index set $\mathcal{I}$, there exists $\gamma:M\to M^{(\mathcal{I})}$ such that $\mu(\rho \gamma)=\beta$ and $\rho \gamma\in\Hom_R(M,A)$. 
%$$\xymatrix{ & & M\ar[d]^\beta \ar@{--{>}}[dll]_\gamma \\ M^{(\mathcal{I})}\ar[r]_-{\rho} & A\ar[r]_\mu &  {B}}$$
%\end{proof}
This results is a module theoretic version of \cite[Theorem 2]{megibbenabsolutely}.
\begin{thm}\label{semifac}
Consider the following conditions for a module $M$:
\begin{enumerate}
\item[(i)] $M$ is a finite $\Sigma$-Rickart module.
\item[(ii)] $\mathfrak{F}_M$ is closed under $M$-pure epimorphisms.
\item[(iii)] If $\mu\in\emph{Hom}_R(A, A')$ is an $M$-pure epimorphism with $A$ $M$-injective, then $A'\in\mathfrak{F}_M$.
\end{enumerate}
Then the implications \emph{(i)$\Rightarrow$(ii)$\Rightarrow$(iii)} hold true. In addition, if $M$ is $M^{(\mathcal{I})}$-projective for any index set $\mathcal{I}$, then the three conditions are equivalent.
\end{thm}
 
\begin{proof} The proofs are similar to those of \cite[Theorem 3.12]{lm1}.
\end{proof}

Remark from \cite[Example 3.13]{lm1} that the converse of (i)$\Rightarrow$(ii) is not true, in general.
Recall that a submodule $N$ of a right $R$-module $M$ is said to be \emph{pure} if for every left $R$-module $K$, the canonical homomorphism $\iota\otimes 1:N\otimes_R K\to M\otimes_R K$ is a monomorphism, where $\iota:N\to M$ is the canonical inclusion. $M$ is said to be \emph{absolutely pure} if $M$ is a pure submodule of any module which contains $M$ as a submodule (see \cite{megibbenabsolutely}). 

\begin{lem}[{\cite[Corollary 2]{megibbenabsolutely} and Remark \ref{frim}(i)}]\label{meggiben}
A right $R$-module $M$ is absolutely pure if and only if $M$ is f-injective if and only if $M$ is in $\mathfrak{F}_R$.
\end{lem}

As a corollary, we have a characterization for right semi-hereditary rings including \cite[Theorem 2]{megibbenabsolutely}.
\begin{cor}
The following conditions are equivalent for a ring $R$:
\begin{enumerate}
\item[(a)] $R$ is a right semi-hereditary ring;
\item[(b)] Every factor module of any f-injective $R$-module is f-injective;
\item[(c)] Every factor module of any absolutely pure $R$-module is absolutely pure;
\item[(d)] Every factor module of any injective $R$-module is  absolutely pure.
\end{enumerate}
\end{cor}

Now, we are going to give a module theoretic version of \cite[Theorem 3.4]{guptaf}.

\begin{thm}\label{semiendo} 
The following conditions are equivalent for a right $R$-module $M$:
\begin{enumerate}
\item[(a)] $M$ is an endoregular module;
\item[(b)] $M$ is a finite $\Sigma$-Rickart module and $M$ is in $\mathfrak{F}_M$;
\item[(c)] $M$ is strongly $D_2$ condition \emph{(}i.e., $M^{(n)}$ has $D_2$ condition for all $n>0$\emph{)} and any finitely $M$-generated submodule of $M^{(n)}$ is a direct summand for all $n>0$.
\end{enumerate}
\end{thm}
\begin{proof}
%After Corollary 2.25
(a)$\Leftrightarrow$(b) Let $M$ be an endoregular module. Then by \cite[Corollary 3.15]{lrr5}, $M^{(n)}$ is an endoregular module, which is Rickart.
Thus, $M$ is finite $\Sigma$-Rickart. Also, from Proposition \ref{endofm} $M$ is in $\mathfrak{F}_M$.
Conversely, let $M\in\mathfrak{F}_M$. Then $M$ has $C_2$ condition from Proposition \ref{fmc2}(iii). Hence $M$ is an endoregular module by \cite[Theorem 3.17]{lrr}.

(a)$\Rightarrow$(c) Since each endoregular module is finite $\Sigma$-Rickart, from Theorem \ref{semihered2} it is easy to see that $M$ is strongly $D_2$ condition. Now, let $N$ be a finitely $M$-generated submodule of $M^{(n)}$ for some $n>0$. 
Then there exists an epimorphism $\rho:M^{(\ell)}\to N$ for some $\ell>0$. On the other hand, let $j:N\to M^{(n)}$ be the canonical inclusion. Then there is a homomorphism $j\rho:M^{(\ell)}\to M^{(n)}$. Therefore $\text{Im}j\rho=N$ is a direct summand of $M^{(n)}$.

%$\rho$ is an $M$-pure epimorphism and $M^{(n)}$ is in $\mathfrak{F}_M$ by Lemma \ref{epipure} and Proposition \ref{prodxi}(ii). Hence $N$ is in $\mathfrak{F}_M$ by Theorem \ref{semifac}. If $j:N\to M^{(\mathcal{I})}$ is the canonical inclusion, there exists a finite subset $\mathcal{F}\subseteq\mathcal{I}$ such that $N\subseteq M^{(\mathcal{F})}$. Since $N$ is in $\mathfrak{F}_M$ this inclusion splits. % by Lemma \ref{monosplxi}.

(c)$\Rightarrow$(b) Let $\varphi:M^{(n)}\to M^{(n)}$ be any endomorphism. By hypothesis $\Img\varphi\dleq M^{(n)}$. Since $M^{(n)}$ has $D_2$ condition, $\Ker\varphi\dleq M^{(n)}$. Thus $M$ is a finite $\Sigma$-Rickart module. 
In addition, let $N$ be a finitely $M$-generated submodule of $M$ and let $\beta:N\to M$ be any homomorphism.  By hypothesis, $N$ is a direct summand of $M$. This implies that $\beta$ can be extended to a homomorphism $\gamma:M\to M$. Thus, $M$ is in $\mathfrak{F}_M$.
\end{proof}

\begin{cor}\label{shabsvn}
The following conditions are equivalent for a ring $R$:
\begin{enumerate}
\item[(a)] $R$ is a von Neumann regular ring;
\item[(b)] $R$ is a right semi-hereditary ring and $R_R$ is an $f$-injective module;
\item[(c)] $R$ is a right semi-hereditary ring and $R_R$ is an absolutely pure module;
\item[(d)] every finitely generated submodule of $R^{(n)}$ is a direct summand for all $n>0$.
\end{enumerate}
\end{cor}

\section{$M$-coherent modules and the endomorphism ring of a finite $\Sigma$-Rickart module}

%\begin{lem}\label{flat}
%The following statements hold true for a module $M$ and $S=\End_R(M)$:
%\begin{enumerate}
%\item[(i)]\cite[15.9]{wisbauerfoundations} $_SM$ is flat if and only if  for every homomorphism $\varphi:M^{(n)}\to M^{(k)}$ with $n,k>0$, $\Ker\varphi$ is $M$-generated.
%\item[(ii)] \cite[12.17]{wisbauerfoundations} $M_R$ is faithfully flat if and only if \ the functor $M_R\otimes\_:R\emph{-Mod}\to \mathbb{Z}\emph{-Mod}$ is exact and reflects zero homomorphisms. 
%\end{enumerate}
%\end{lem} 

%\begin{lem}\label{flat1}
%If $_SM$ is flat where $S=\End_R(M)$ then $\text{Ker}\varphi=\mathbf{r}_S(\varphi)M$.
%\end{lem} 
%\begin{proof} Note that $\mathbf{r}_S(\varphi)=\text{Hom}_R(M, \text{Ker}\varphi)$. Also, $\text{Ker}\varphi=\text{Hom}_R(M, \text{Ker}\varphi)M$ because $_SM$ is flat.
%\end{proof}

The next result can be seen as a generalization of Schanuel's Lemma \cite[5.1]{l}.

\begin{lem}\label{schanuel}
Let $M$ be a right $R$-module and $K\in\add(M)$. Let
\[0\longrightarrow D\overset{\sigma}{\longrightarrow} K \overset{\rho}{\longrightarrow} C\longrightarrow 0\]
\[0\longrightarrow A\overset{\alpha}{\longrightarrow} B\overset{\beta}{\longrightarrow} C\longrightarrow 0\]
be short exact sequences with $\beta$ an $M$-pure epimorphism. Then there exists a short exact sequence 
\[0\longrightarrow D\overset{\delta}{\longrightarrow} K\oplus A\overset{\eta}{\longrightarrow} B\longrightarrow 0.\]
Moreover, if $\rho$ is also an $M$-pure epimorphism, then so is $\eta$.
\end{lem}

\begin{proof}
Consider the following diagram:
\[\xymatrix{0 \ar[r] & D\ar[r]^{\sigma} & K \ar[r]^\rho \ar@{--{>}}[d]^\gamma & C \ar[r] \ar[d]^= & 0 \\ 
0 \ar[r] & A \ar[r]_\alpha & B \ar[r]_\beta & C \ar[r] & 0.}\]
Since $K\in\add(M)$ and $\beta$ is an $M$-pure epimorphism, there exists $\gamma:K\to B$ such that $\rho=\beta\gamma$. We claim that the following sequence is exact
\begin{equation}\label{sequence}
\xymatrix{0\ar[r] & D\ar[r]^-\delta & K\oplus \Img\alpha \ar[r]^-\eta & B\ar[r] & 0}
\end{equation}
where $\delta(d)=(\sigma(d), \gamma(\sigma(d)))$ and $\eta(k,x)=\gamma(k)-x$ for $d\in D$, $k\in K$ and $x\in \Img\alpha$: 
It is clear that $\gamma(\sigma(d))\in\Img\alpha$, $\delta$ is a monomorphism, and $\eta\delta=0$. Let $(k,x)\in K \oplus \Img\alpha$ such that $\eta(k,x)=0$. That is, $\gamma(k)=x$. Then $0=\beta(x)=\beta(\gamma(k))=\rho(k)$.
Since $\Img\sigma=\Ker\rho$ there exists $d\in D$ such that $\sigma(d)=k$. Thus $\delta(d)=(\sigma(d),\gamma(\sigma(d)))=(k,\gamma(k))=(k,x)$. Hence $\Img\delta=\Ker\eta$.
Now, it remains to show that $\eta$ is an epimorphism. Let $b\in B$. Since $\rho$ is an epimorphism, there exists $\ell\in K$ such that $\rho(\ell)=\beta(b)$. Hence $\beta(\gamma(\ell)-b)=0$, that is, there exists $a\in A$ such that $\alpha(a)=\gamma(\ell)-b$. Thus $\eta(\ell,\alpha(a))=b$, proving the claim. Since $A\cong\Img\alpha$, we have an exact sequence
\[0\to D\to K\oplus A\to B\to 0.\]
Moreover, suppose $\rho$ is also an $M$-pure epimorphism. Let $\zeta:M\to B$ be any homomorphism. Then, $\beta\zeta:M\to C$. Since $\rho$ is an $M$-pure epimorphism, there exist $\varphi:M\to K$ such that $\rho\varphi=\beta\zeta$. This implies that $\beta\gamma\varphi=\beta\zeta$ and so $\gamma\varphi(m)-\zeta(m)\in\Ker\beta=\Img\alpha$ for all $m\in M$. Define $\overline{\varphi}:M\to K\oplus\Img\alpha$ as $\overline{\varphi}(m)=(\varphi(m),\gamma\varphi(m)-\zeta(m))$. It is clear that $\eta\overline{\varphi}=\zeta$.
\end{proof}

\begin{lem}\label{mgenext}
Let $M$ be a right $R$-module and 
\[0\longrightarrow A\overset{\alpha}{\longrightarrow} B\overset{\beta}{\longrightarrow} C\longrightarrow 0\]
be an exact sequence with $A$ and $C$ finitely $M$-generated modules. If $\beta$ is an $M$-pure epimorphism then $B$ is finitely $M$-generated.
\end{lem}

\begin{proof}
Since $C$ is finitely $M$-generated, 
there exists an epimorphism $\rho:M^{(n)}\to C$ for some $n\in\mathbb{N}$. Consider the pull-back $P$ of $(\beta,\rho)$, that is, $P=\{(b,m)\in B\oplus M^{(n)}\mid \beta(b)=\rho(m)\}$:
\[\xymatrix{ & & P\ar[d]_{\rho'} \ar[r]^{\beta'} & M^{(n)}\ar[d]^\rho & \\ 0\ar[r] & A \ar[r]_\alpha & B \ar[r]_\beta & C\ar[r] & 0.}\]
Note that both $\rho'$ and $\beta'$ are epimorphisms. On the other hand, $\Ker\beta'=\Ker\beta\oplus 0\cong A$. Since $\beta$ is an $M$-pure epimorphism, there exists $\kappa:M^{(n)}\to B$. Therefore, $\beta'$ splits. Hence $P\cong A\oplus M^{(n)}$. This implies that $P$ is finitely $M$-generated because $A$ is finitely $M$-generated.  Since $\rho'$ is an epimorphism, $B$ is finitely $M$-generated.
\end{proof}

\begin{prop}\label{kermpre1}
Let $C$ be a module such that there exists an exact sequence $0\to D\to M^{(n)}\overset{\pi}{\rightarrow} C\to 0$ with $\pi$ an $M$-pure epimorphism and $D$ finitely $M$-generated. If $\rho:B\to C$ is an $M$-pure epimorphism with $B$ finitely $M$-generated, then $\Ker\rho$ is finitely $M$-generated.
\end{prop}

\begin{proof}
Consider the exact sequence
\[0\to D\to M^{(n)}\overset{\pi}{\rightarrow} C\to 0\]
with $D$ finitely $M$-generated. Since $\rho$ and $\pi$ are $M$-pure epimorphisms, by Lemma \ref{schanuel} we get an exact sequence
\[0\to D\to M^{(n)}\oplus \Ker\rho\overset{\eta}{\rightarrow} B\to 0\]
with $\eta$ an $M$-pure epimorphism. From Lemma \ref{mgenext}, $\Ker\rho$ is finitely $M$-generated.
\end{proof}

%\begin{lem}
%Let $M$ be a module such that $\Ext^1_R(M,M)=0$. Let $\rho:M^{(n)}\to N$ be any epimorphism. Then, $\rho$ is an $M$-pure epimorphism if and only if $\Ext^1_R(M,\Ker\rho)=0$.
%\end{lem}
%
%\begin{proof}
%Consider the exact sequence $0\to \Ker\rho\overset{i}{\rightarrow} M^{(n)}\overset{\rho}{\rightarrow} N\to 0$. Applying $\Hom_R(M,-)$, we get an exact sequence:
%\[\xymatrix{\ar[r] & \Hom_R(M,M^{(n)})\ar[r]^{\rho_\ast} & \Hom_R(M,N)\ar[r] & \Ext^1_R(M,\Ker\rho)\ar[r] & \Ext^1_R(M,M^{(n)})=0.}\]
%Hence, $\rho$ is an $M$-pure epimorphism if and only if $\Ext^1_R(M,\Ker\rho)=0$.
%\end{proof}

Recall that a right $R$-module $M$ is said to be \emph{intrinsically projective} if for every diagram
\[\xymatrix{ & M\ar[d]^\beta \ar@{--{>}}[dl]_{\gamma} & \\ M^{(n)}\ar[r]_\alpha & N\ar[r] & 0}\]
with $n>0$ and $N\leq M$, there exists $\gamma:M\to M^{(n)}$ such that $\alpha\gamma=\beta$ (see \cite{wr}).
 Note that every finite $\Sigma$-Rickart module and every quasi-projective module is intrinsically projective.
In addition, a right $R$-module $M$ is intrinsically projective  if and only if $I=\Hom_R(M,IM)$ for all finitely generated right ideals $I\leq \End_R(M)$ (\cite[5.7]{wr}).

\begin{lem}\label{inproympe}
Let $M$ be an intrinsically projective module. Then the following statements hold true:
\begin{enumerate}
\item[(i)] Any epimorphism $\rho:L\to C$, with $C\leq M$ and $L$ finitely $M$-generated, is an $M$-pure epimorphism.
\item[(ii)] For any finitely $M$-generated submodule $N$ of $M^{(n)}$ with any $n\in\mathbb{N}$, $\Hom_R(M,N)$ is a finitely generated right $S$-module where $S=\End_R(M)$.
\end{enumerate}
\end{lem}

\begin{proof}
(i) Let $\rho:L\to C$ be any epimorphism with $C\leq M$ and $L$ finitely $M$-generated. Let $\alpha:M\to C$ be any homomorphism. Since $L$ is finitely $M$-generated there exists an epimorphism $\beta:M^{(n)}\to L$ for some $n>0$. Since $M$ is intrinsically projective, there exists $\gamma:M\to M^{(n)}$ such that $\alpha=(\rho\beta)\gamma=\rho(\beta\gamma)$.

(ii) Let $N\leq M^{(n)}$ be finitely $M$-generated. Hence there exist an integer $k>0$ and an epimorphism $\rho:M^{(k)}\to N$. Let $\ell=\text{max}\{k,n\}$, then we can see $\rho:M^{(\ell)}\to N$ and $N\leq M^{(\ell)}$. 
Let $\pi_i:M^{(\ell)}\to M$ denote the canonical projection for each $1\leq i\leq \ell$. Let $\varphi:M\to N$ be any homomorphism and consider the epimorphisms $\pi_i\rho:M^{(\ell)}\to \pi_i(N)$ for $1\leq i\leq\ell$. Since $M$ is intrinsically projective there exists $\gamma_i:M\to M^{(\ell)}$ such that $\pi_i\rho\gamma_i=\pi_i\varphi$ for all $1\leq i\leq \ell$. Define $\gamma:M\to M^{(\ell)}$ as $\gamma(m)=(\gamma_1(m),\dots,\gamma_\ell(m))$. Hence $\rho\gamma=\varphi$. Let $\eta_i:M\to M^{(\ell)}$ denote the canonical inclusion for each $1\leq i\leq \ell$. Then
\[\varphi=\rho\gamma=\sum_{i=1}^\ell(\rho\eta_i)\pi_i\gamma ~\text{for}~ \pi_i\gamma\in S.\]
Thus, $\Hom_R(M,N)$ is generated by $\rho\eta_1,\rho\eta_2,\dots,\rho\eta_\ell$.
\end{proof}

For a right $R$-module $M$, a right $R$-module $N$ is said to be \emph{finitely} $M$-\emph{presented} if there exists an exact sequence 
$M^{(\ell)}\to M^{(n)}\to N\to 0$ for some $n,\ell>0$ 
(\cite{wisbauerfoundations}).

\begin{lem}\label{kermpre}
Let $M$ be an intrinsically projective module and $C$ be a finitely $M$-presented submodule of $M$. If $\rho:B\to C$ is an epimorphism with $B$ finitely $M$-generated, then $\Ker\rho$ is finitely $M$-generated.
\end{lem}

\begin{proof}
Since $C$ is finitely $M$-presented there exists an exact sequence
\[0\to D\to M^{(n)}\overset{\pi}{\rightarrow} C\to 0\]
with $D$ finitely $M$-generated. Note that $\rho$ and $\pi$ are $M$-pure epimorphisms from Lemma \ref{inproympe}(i). Therefore, the result follows from Proposition \ref{kermpre1}.
\end{proof}

%\begin{cor}\label{kermpre2}
%Let $M$ be a quasi-projective module and $C$ be a finitely $M$-presented module. If $\rho:B\to C$ is an epimorphism with $B$ finitely $M$-generated, then $\Ker\rho$ is finitely $M$-generated.
%\end{cor}
%
%\begin{proof}
%Since $C$ is finitely $M$-presented there exists an exact sequence
%\[0\to D\to M^{(n)}\overset{\pi}{\rightarrow} C\to 0\]
%with $D$ finitely $M$-generated. Note that $\rho$ and $\pi$ are $M$-pure epimorphisms because $M$ is quasi-projective. Therefore, the result follows from Proposition \ref{kermpre1}
%\end{proof}

\begin{prop}\label{intersum}
Let $M$ be an intrinsically projective module and $A,B\leq M$ be finitely $M$-presented submodules. Consider the following exact sequence
\[0\longrightarrow A\cap B \longrightarrow A\oplus B\overset{\pi}{\longrightarrow} A+B\longrightarrow 0.\]
Then $A+B$ is finitely $M$-presented if and only if $A\cap B$ is finitely $M$-generated.
\end{prop}
\begin{proof}
Let $A$ and $B$ be finitely $M$-presented submodules of $M$. Then there exist epimorphisms $\rho_1:M^{(n_1)}\to A$ and $\rho_2:M^{(n_2)}\to B$ for some $n_1, n_2\in\mathbb{N}$. So, $\rho=\rho_1\oplus\rho_2:M^{(n_1)}\oplus M^{(n_2)}\to A\oplus B$ is an epimorphism. That is, $A\oplus B$ is finitely $M$-generated. 

Suppose $A+B$ is finitely $M$-presented. Since $M$ is intrinsically projective and $A\oplus B$ is finitely $M$-generated, $A\cap B$ is finitely $M$-generated by Lemma \ref{kermpre}.

Conversely, since $A$ and $B$ are finitely $M$-presented, there is an exact sequence
\[0\to \Ker\rho \to M^{(n_1)}\oplus M^{(n_2)}\overset{\rho}{\rightarrow} A\oplus B \to 0\]
with $\Ker\rho$ finitely $M$-generated and $\rho=\rho_1\oplus\rho_2$ an $M$-pure epimorphism by Lemma \ref{inproympe}(i) and Remark \ref{mpepisprop}(iii). Consider the exact sequence
\[0\to \Ker\pi\rho\to M^{(n_1+n_2)}\overset{\pi\rho}{\rightarrow} A+B\to 0.\]
Note that $\pi$ is an $M$-pure epimorphism by Lemma \ref{inproympe}(i) because $A\oplus B$ is finitely $M$-generated. Hence from Lemma \ref{schanuel}, we have an exact sequence
\[0\to \Ker\pi\rho \to (A\cap B)\oplus M^{(n_1+n_2)}\overset{\eta}{\rightarrow} A\oplus B\to 0\]
with $\eta$ an $M$-pure epimorphism. Since $A\cap B$ is finitely $M$-generated, $(A\cap B)\oplus M^{(n_1+n_2)}$ is finitely $M$-generated. 
Also since $\Ker\rho$ is finitely $M$-generated, $\Ker\eta\cong\Ker\pi\rho$ is finitely $M$-generated by Proposition \ref{kermpre1}. This implies that $A+B$ is finitely $M$-presented. 
\end{proof}

\begin{defn}
Let $M$ be a right $R$-module and $N$ be a finitely $M$-generated module. The module $N$ is called $M$-\emph{coherent} if for any $n>0$ and every homomorphism $\rho:M^{(n)}\to N$, $\Ker\rho$ is finitely $M$-generated.
\end{defn}

Remark that if a right $R$-module $M$ is $M$-coherent then $_SM$ is flat where $S=\End_R(M)$. Also, a ring $R$ is said to be \emph{right coherent} if $R_R$ is $R$-coherent. In addition,  $M$ is a coherent right $R$-module if and only if $M$ is an $R$-coherent right $R$-module (\cite[4G]{l}).
%{ It is easy to see that the $\mathbb{Z}$-module $\mathbb{Z}_4$ is $\mathbb{Z}_4$-coherent but it is not a finite $\Sigma$-Rickart module.}

\begin{lem}\label{cohepikfm} Let $M$ be an $M$-coherent module.  If $\rho:B\to C$ is an epimorphism with $B$ finitely $M$-generated and $C\leq M$, then $\Ker\rho$ is finitely $M$-generated.
\end{lem}

\begin{proof} Let $\rho:B\to C$ be an epimorphism with $B$ finitely $M$-generated and $C\leq M$. Hence there exists an epimorphism $\pi:M^{(n)}\to B$ for some $n>0$. Since $M$ is $M$-coherent, $\Ker\rho\pi=\pi^{-1}(\Ker\rho)$ is finitely $M$-generated. Since $\Ker\rho$ is a factor module of $\Ker\rho\pi$, $\Ker\rho$ is finitely $M$-generated. 
\end{proof}

\begin{prop}\label{cohepikfm3} $M$ is an $M$-coherent module if and only if every finitely $M$-generated submodule of $M$ is finitely $M$-presented.
\end{prop}

\begin{proof} It directly follows from Lemma \ref{cohepikfm} and the definition of an $M$-coherent module.
\end{proof}

\begin{thm}\label{coh3}
Consider the following conditions for a module $M$:
\begin{enumerate}
\item[(i)] $M$ is an $M$-coherent module.
\item[(ii)] The intersection of two finitely $M$-generated submodules of $M$ is finitely $M$-generated and $\Ker\varphi$ is finitely $M$-generated for all $\varphi\in\End_R(M)$.
\end{enumerate}
Then \emph{(i)}$\Rightarrow$\emph{(ii)} holds. In addition, if $M$ is intrinsically projective then the two conditions are equivalent.
\end{thm}

\begin{proof}
(i)$\Rightarrow$(ii) By the definition of an $M$-coherent module, $\Ker\varphi$ is finitely $M$-generated for every $\varphi\in\End_R(M)$. Now, let $A$ and $B$ be finitely $M$-generated submodules of $M$. Hence $A\oplus B$ is finitely $M$-generated. Consider the natural exact sequence
\[0\to A\cap B \to A\oplus B\rightarrow A+B\to 0.\]
It follows from Lemma \ref{cohepikfm}  that $A\cap B$ is finitely $M$-generated.

In addition, suppose $M$ is intrinsically projective. For (ii)$\Rightarrow$(i), we are going to prove by induction on $n$ that the kernel of any homomorphism $\varphi:M^{(n)}\to M$ is finitely $M$-generated. If $n=1$, the kernel of any $\varphi\in\End_R(M)$ is finitely $M$-generated by hypothesis. Suppose $n>1$ and for all homomorphisms $\rho:M^{(\ell)}\to M$ with $\ell<n$, $\Ker\rho$ is finitely $M$-generated. Let $\psi:M^{(n)}\to M$ be any homomorphism. We have that $\Img\psi=\psi(0\oplus M^{(n-1)})+\psi(M\oplus 0)$. 
Since $\Ker(\psi|_{0\oplus M^{(n-1)}})$ is finitely $M$-generated by the induction hypothesis and $\Ker(\psi|_{M\oplus 0})$ is finitely $M$-generated, $\psi(0\oplus M^{(n-1)})$ and $\psi(M\oplus 0)$ are finitely $M$-presented. By hypothesis, $\psi(0\oplus M^{(n-1)})\cap\psi(M\oplus 0)$ is finitely $M$-generated. It follows from Proposition \ref{intersum} that $\Img\psi=\psi(0\oplus M^{(n-1)})+\psi(M\oplus 0)$ is finitely $M$-presented. Hence $\Ker\psi$ is finitely $M$-generated by Lemma \ref{kermpre}. Thus, $M$ is an $M$-coherent module.
\end{proof}

\begin{cor}[{\cite[Corollary 4.60]{l}}]
A ring $R$ is a right coherent ring if and only if the intersection of two finitely generated ideals of $R$ is finitely generated and $\mathbf{r}_R(a)$ is finitely generated for all $a\in R$.
\end{cor}

%\begin{thm}
%The following conditions are equivalent for an intrinsically projective module $M$:
%\begin{enumerate}
%\item[(a)] $M$ is $M$-coherent;
%\item[(b)] the intersection of two finitely $M$-generated submodules of $M$ is finitely $M$-generated and $\Ker\varphi$ is finitely $M$-generated for all $\varphi\in\End_R(M)$.
%\end{enumerate}
%\end{thm}

\begin{cor}\label{iprick}
Let $M$ be an intrinsically projective Rickart module. Then $M$ is $M$-coherent if and only if the intersection of two finitely $M$-generated submodules of $M$ is finitely $M$-generated. 
\end{cor}

Chase (\cite[4.60]{l}) shows that a domain $R$ is right coherent if and only if the intersection of two finitely generated right ideals of $R$ is finitely generated. In the next result, we extend to a right Rickart ring.
\begin{cor}
A right Rickart ring $R$ is right coherent if and only if the intersection of two finitely generated right ideals of $R$ is finitely generated.
\end{cor}

%Some conditions on the endomorphism ring of a module are related with the property of flatness. For the convenience of the reader we write the following lemmas. 
\begin{lem}[{\cite[15.9]{wisbauerfoundations}}]\label{flat}
The following conditions are equivalent for a module $M$:
\begin{enumerate}
\item[(a)] $_SM$ is flat where $S=\End_R(M)$;
\item[(b)] for any homomorphism $\rho:M^{(n)}\to M^{(k)}$ with $n,k>0$, $\Ker\rho$ is $M$-generated;
\item[(c)] for any homomorphism $\rho:M^{(n)}\to M$ with $n>0$, $\Ker\rho$ is $M$-generated.
\end{enumerate}
\end{lem}

Note that if a module $N$ is $M$-generated then $N=\text{Hom}_R(M, N)M$.
For an intrinsically projective module $M$, there is another characterization when $M$ is $M$-coherent as well as Theorem \ref{coh3}.
\begin{thm}\label{coh}
Consider the following conditions for a module $M$:
\begin{itemize}
\item[(i)] $S$ is a right coherent ring and $_SM$ is flat.
\item[(ii)] $M$ is an $M$-coherent module.
\end{itemize}
Then \emph{(i)}$\Rightarrow$\emph{(ii)} holds. In addition, if $M$ is intrinsically projective then the two conditions are equivalent.
\end{thm}

\begin{proof}
(i)$\Rightarrow$(ii) Let $N\leq M$ be finitely $M$-generated. Consider the exact sequence $0\to K\to M^{(n)}\overset{\rho}{\rightarrow} N\to 0$. Applying $\Hom_R(M,\_)$, we get
\[0\to \Hom_R(M,K)\to \Hom_R(M,M^{(n)})\overset{\rho_{*}}{\rightarrow} \Hom_R(M,N).\] 
Since $N\leq M$, $\Hom_R(M,N)$ embeds in $S$. Note that $\Img\rho_{*}$ is finitely generated as a right $S$-module. This implies that $\Hom_R(M,K)$ is finitely generated as a right $S$-module because $S$ is a right coherent ring. Hence there exists an epimorphism $S^{(\ell)}\to \Hom_R(M,K)$ for some $\ell>0$. Note that $K$ is $M$-generated because $_SM$ is flat from Lemma \ref{flat}. Applying $\_\otimes_SM$,
\[\xymatrix{S^{(\ell)}\otimes_SM \ar[r] \ar[d]^\cong & \Hom_R(M,K)\otimes_SM \ar[r] \ar[d]^\cong & 0 \\
M^{(\ell)} \ar[r] & K \ar[r] & 0.}\]
Thus $K$ is finitely $M$-generated. This implies that $M$ is $M$-coherent.
\vspace{0.1cm}

In addition, suppose $M$ is intrinsically projective. For (ii)$\Rightarrow$(i), it is easy to see that $_SM$ is flat from Lemma \ref{flat}. Let $I$ be a finitely generated right ideal of $S$. Then there is an exact sequence 
\[0\to \Ker\eta\to S^{(n)}\overset{\eta}{\rightarrow} I\to 0\]
for some $n>0$. 
It is enough to show that $J=\Ker\eta$ is finitely generated as a right $S$-module. Applying the functor $\_\otimes_SM$, since $_SM$ is flat we get
\begin{equation}\label{3}
\xymatrix{0 \ar[r] & J\otimes_SM \ar[r]\ar[d]^\cong & S^{(n)}\otimes_SM \ar[r]^{\eta\otimes 1} \ar[d]_\alpha^\cong & I\otimes_S M \ar[r]\ar[d]_\beta^\cong & 0 \\
0 \ar[r] & J' \ar[r] & M^{(n)} \ar[r] & IM \ar[r] & 0,}
\end{equation}
where $\alpha,\beta$ are the canonical homomorphisms and $J'=\Ker\beta(\eta\otimes 1)\alpha^{-1}$. Since $M$ is intrinsically projective and $IM\leq M$, the functor $\Hom_R(M,\_)$ is exact in (\ref{3}). Therefore, the following diagram has exact rows:
\[\xymatrix{0 \ar[r] & J \ar[d] \ar[r] & S^{(n)} \ar[d]^\cong \ar[r]^\eta & I \ar[d]^\cong \ar[r] & 0\\ 0 \ar[r] & \Hom_R(M, J\otimes_SM) \ar[d]^\cong \ar[r] & \Hom_R(M,S^{(n)}\otimes_SM) \ar[d]^\cong \ar[r] & \Hom_R(M,I\otimes_S M) \ar[d]^\cong \ar[r] & 0 \\ 0 \ar[r] & \Hom_R(M, J') \ar[r] & \Hom_R(M,M^{(n)}) \ar[r] & \Hom_R(M,IM) \ar[r] & 0.}\]

Hence $J\cong\Hom_R(M,J')$. Since $M$ is $M$-coherent, $J'$ is finitely $M$-generated. Therefore $J$ is a finitely generated right $S$-module by Lemma \ref{inproympe}(ii). 
\end{proof}

\begin{cor}\label{coh5}
The following are equivalent for an intrinsically projective module $M$:
\begin{itemize}
\item[(a)] $M$ is an $M$-coherent module;
\item[(b)] $S=\End_R(M)$ is a right coherent ring and $_SM$ is flat;
\item[(c)] The intersection of two finitely $M$-generated submodules of $M$ is finitely $M$-generated and $\Ker\varphi$ is finitely $M$-generated for all $\varphi\in\End_R(M)$.
\end{itemize}
\end{cor}

\begin{prop}\label{manysemiher}
Let $M$ be a finite $\Sigma$-Rickart module. Then the following statements hold true:
\begin{enumerate}
\item[(i)] $\End_R(M)$ is a right semi-hereditary ring.
\item[(ii)] $\End_R(M)$ is a right coherent ring.
\item[(iii)] Every finitely $M$-generated submodule of $M$ is $M$-coherent.
\end{enumerate}
\end{prop}
\begin{proof}
(i) Since $M^{(n)}$ is Rickart, $\mathsf{Mat}_n(S)$ is a right Rickart ring for all $0<n\in\mathbb{N}$ with $S=\End_R(M)$. Then $S$ is a right semi-hereditary ring by \cite[Proposition]{s}.
(ii) It is trivial (see Lemma \ref{Chase}).

%(iii) Let $I$ be a finitely generated right ideal of $S$, generated by $\varphi_1,\varphi_2,\dots,\varphi_n\in S$. We have an epimorphism $\rho:M^{(n)}\to IM$ given by $\rho(m_1,m_2,\dots,m_n)=\sum_{i=1}^n\varphi_i(m_i)$. Since $M$ is finite $\Sigma$-Rickart, $\rho$ splits. Therefore, applying $\Hom_R(M,\_)$ we get an epimorphism
%\[S^{(n)}\cong\Hom_R(M,M^{(n)})\overset{\rho_*}{\rightarrow}\Hom_R(M,IM)\to 0.\]
%For $1\leq i\leq n$, let $e_i\in S^{(n)}$ denote the element which has $1$ in the $i$th position and $0$ elsewhere. Then $\rho_{*}(e_i)=\varphi_i$. Therefore $\Hom_R(M,IM)$ is generated by $\varphi_1,\varphi_2,\dots,\varphi_n$. Thus $I=\Hom_R(M,IM)$.

%(iv) Let $\varphi:M^{(n)}\to M^{(k)}$ be any homomorphism. Then $\Ker \varphi\dleq M^{(n)}$ from Lemma \ref{relrick}(i). This implies that $\Ker \varphi$ is $M$-generated. By Lemma \ref{flat}, $_SM$ is flat.
%
%On the other hand, let $N_S$ be a right $S$-module such that $N\otimes_SM=0$. Let $n\in N$ be arbitrary. Then $nS\otimes_SM=0$ because $_SM$ is flat. We have the exact sequence
%\[0\to \mathbf{r}_S(n)\to S\to nS\to 0\]
%Note that $\mathbf{r}_S(n)$ is finitely generated because $S$ is right coherent. Since $(S/\mathbf{r}_S(n))\otimes_SM\cong M/\mathbf{r}_S(n)M$ and  $nS\otimes_SM=0$, $\mathbf{r}_S(n)M=M$. Hence $\mathbf{r}_S(n)=S$ by (iii). Thus $nS=0$ for all $n\in N$. Therefore $_SM$ is faithfully flat.

(iii) Let $N$ be a finitely $M$-generated submodule of $M$ and let $\varphi:M^{(n)}\to N$ be any homomorphism. Since $N\leq M$ and $M$ is finite $\Sigma$-Rickart, $\Ker\varphi\dleq M^{(n)}$. Hence $\Ker\varphi$ is finitely $M$-generated.
\end{proof}

\begin{lem}[{\cite[39.10(2)]{wisbauerfoundations}}] \label{srick} If $S=\End_R(M)$ is a right Rickart ring then $\mathbf{r}_S(\varphi)M\dleq M$ for all $\varphi\in S$.
\end{lem}

\begin{lem}[{Chase \cite[Theorem 4.1]{c1}}]\label{Chase} A ring $R$ is  right semi-hereditary
if and only if $R$ is a right coherent ring and all right ideals of $R$ are flat.
\end{lem}

As a finitely generated $\Sigma$-Rickart module is characterized in terms of its endomorphism ring (Theorem \ref{endoringh}), we obtain the characterization of a finite $\Sigma$-Rickart module using its endomorphism ring.
\begin{thm}\label{endoringsh}
The following conditions are equivalent for a module $M$ and $S=\End_R(M)$:
\begin{enumerate}
\item[(a)] $M$ is a finite $\Sigma$-Rickart module;
\item[(b)] $S$ is a right semi-hereditary ring and $_SM$ is flat;
\item[(c)] $M$ is an intrinsically projective $M$-coherent module and all right $S$-ideals are flat.
\end{enumerate}
\end{thm}

\begin{proof}      
(a)$\Rightarrow$(b) It follows from Proposition \ref{manysemiher}(i) and Lemma \ref{flat}.

(b)$\Rightarrow$(c) Since $S$ is a right semi-hereditary ring, $S$ is a right coherent ring and every right $S$-ideal is flat by Lemma \ref{Chase}. 
It follows from \cite[Examples 5.6(2)]{wr} that $M$ is intrinsically projective. From Theorem \ref{coh} $M$ is $M$-coherent because $_SM$ is flat. 

(c)$\Rightarrow$(a) Since $M$ is intrinsically projective and $M$-coherent, by Theorem \ref{coh} $S$ is a right coherent ring and $_SM$ is flat. 
 From Lemma \ref{Chase} $S$ is a right semi-hereditary ring because all right $S$-ideals are flat. Let $\varphi:M^{(n)}\to M^{(n)}$ be any endomorphism. 
Since $_SM$ is flat as above, $\Ker\varphi$ is $M$-generated by Lemma \ref{flat}.
Hence $\Ker\varphi=\Hom_{R}(M^{(n)}, \Ker\varphi )M^{(n)}=\mathbf{r}_{\mathsf{Mat}_n(S)}(\varphi)M^{(n)}\dleq M^{(n)}$ from Lemma \ref{srick}.
 Therefore $M$ is a finite $\Sigma$-Rickart module.
%Consider the exact sequence
%\[0\to \Ker\varphi \overset{\iota}{\rightarrow} M^{(n)}\overset{\varphi}{\rightarrow} \Img\varphi\to 0\]
%Applying the functor $\Hom_R(M,\_)$ we get
%\[0\to \Hom_R(M,\Ker\varphi) \overset{\iota_{*}}{\rightarrow} \Hom_R(M,M^{(n)}) \overset{\varphi_{*}}{\rightarrow} \Hom_R(M,\Img\varphi)\]
%Since $\Img\varphi\leq M^{(n)}$, we can see $\Hom_R(M,\Img\varphi)$ as a submodule of $S^{(n)}$. Note that $\Img \varphi_{*}$ is finitely generated right $S$-module. Since $S$ is a right semi-hereditary ring, $\Img\varphi_{*}$ is projective. Therefore  $\Ker\varphi_{*}=\iota_{*}(\Hom_R(M,\Ker\varphi))$ is a direct summand of $\Hom_R(M,M^{(n)})$, that is,
%\[\Hom_R(M,M^{(n)})=\iota_{*}(\Hom_R(M,\Ker\varphi))\oplus A_S\]
%for some $A_S\leq \Hom_R(M,M^{(n)})$. Now, tensoring with $\_\otimes_SM$
%\[\Hom_R(M,M^{(n)})\otimes_S M=\left(\iota_{*}(\Hom_R(M,\Ker\varphi))\otimes_S M\right)\oplus \left(A\otimes_S M\right)\]
%There is a canonical isomorphism $\Hom_R(M,M^{(n)})\otimes_S M\cong M^{(n)}$ given by evaluation. Since $\Ker\varphi$ is $M$-generated by Lemma \ref{flat}, under this isomorphism
%\[M^{(n)}=\Ker\varphi\oplus N\]
%for some $N\leq {M_R}$. Thus $M$ is a finite $\Sigma$-Rickart module. 
\end{proof}

%Note that if $_SM$ is not flat, the implication (b)$\Rightarrow$(a) may not be true. For, consider the following example.

The ``$_SM$ is flat" condition in (b)$\Rightarrow$(a) is not superfluous as shown next.
\begin{exam} (i) Consider  $\mathbb{Z}_{p^\infty}$ as a $\mathbb{Z}$-module. 
Then $S=\text{End}_\mathbb{Z}({\mathbb{Z}_{p^\infty}})$ is a right semi-hereditary ring. But $\mathbb{Z}_{p^\infty}$ is neither a finite $\Sigma$-Rickart $\mathbb{Z}$-module nor a flat left $S$-module.\\
(ii) The $\mathbb{Z}$-module $\mathbb{Z}_4$ is $\mathbb{Z}_4$-coherent, however $\mathbb{Z}_4$ is not finite $\Sigma$-Rickart.
\end{exam}

An explicit application of Theorem \ref{endoringsh} is exhibited in the next example. 
\begin{exam}
Let $R$ be the ring of $n\times n$ upper triangular matrices over a right semi-hereditary ring $A$. Let $e\in R$ be a unit matrix with $1$ in the $(1,1)$-position and $0$ elsewhere. 
Then $\End_R(eR)\cong A$ and $eR\cong A^{(n)}$ as projective left $A$-modules. Therefore $eR$ is a finite $\Sigma$-Rickart module by Theorem \ref{endoringsh}. For example, while $R=\left(\begin{smallmatrix}\mathbb{Z}&\mathbb{Z}\\0&\mathbb{Z}\end{smallmatrix}\right)$ is not a right hereditary ring, $eR=\left(\begin{smallmatrix}\mathbb{Z}&\mathbb{Z}\\0&0\end{smallmatrix}\right)$ is a finite $\Sigma$-Rickart $R$-module for $e=\left(\begin{smallmatrix}1&0\\0&0\end{smallmatrix}\right)$.
\end{exam}

Since every finitely generated projective module over a right semi-hereditary ring is a finite $\Sigma$-Rickart module, its endomorphism ring is a right semi-hereditary ring as a consequence of Theorem \ref{endoringsh}.

\begin{cor} The following statements hold true:
\begin{enumerate}
\item[(i)] \emph{(\cite[Theorem 2.10]{genqf3})} If $R$ is a right semi-hereditary ring and $P$ is a finitely generated projective $R$-module, then $\End_R(P)$ is a right semi-hereditary ring.
\item[(ii)] If $R$ is a right semi-hereditary ring, so is $eRe$ for any idempotent $e\in R$.
\end{enumerate}
\end{cor}

\section{Applications}

\begin{prop}\label{emsinj}
Let $M$ be a right $R$-module with $S=\End_R(M)$ such that $_SM$ is flat. Then the following equivalences hold true: 
\begin{enumerate}
\item[(i)] A right $R$-module $A$ is in $\mathfrak{E}_M$ iff $\Hom_R(M,A)$ is an injective right $S$-module.
\item[(ii)]  A right $R$-module $A$ is in $\mathfrak{F}_M$ iff $\Hom_R(M,A)$ is an f-injective right $S$-module.
\end{enumerate}
\end{prop}

\begin{proof}
(i) Suppose $A\in\mathfrak{E}_M$. Let $I_S$ be a right ideal of $S$ and let $\alpha:I\to \Hom_R(M,A)$ be any $S$-homomorphism. Hence we have the following diagram of right $R$-modules
\[\xymatrix{0\ar[r] & I\otimes_S M\ar[r]^{\iota\otimes 1} \ar[d]_{\alpha\otimes 1} & S\otimes_SM\cong M\ar@{--{>}}[ddl]^{g} \\ &  \Hom_R(M,A)\otimes_S M \ar[d]_j &  \\ &  A & }\]
where $\iota:{I_S}\to S$ is the canonical inclusion and $j:\Hom_R(M,A)\otimes_S M\to A$ is given by $j(f\otimes m)=f(m)$. Note that $I\otimes_S M$ is $M$-generated and since $_SM$ is flat, $\iota\otimes 1$ is a monomorphism. Let $\theta$ denote the canonical isomorphism $S\otimes_S M\to M$. By the definition of $\mathfrak{E}_M$, there exists an $R$-homomorphism $g:M \to A$ such that $g\theta(\iota\otimes 1)=j(\alpha\otimes 1)$. Define $\bar{\alpha}:S\to \Hom_R(M,A)$ as $(\bar{\alpha}(f))(m)=gf(m)$ for $f\in S$. Let $h\in I$ and $m\in M$. Hence 
\[(\bar{\alpha}\iota(h))(m)=gh(m)=g(\theta(h\otimes m))=(g\theta(\iota\otimes 1))(h\otimes m)\]
\[=j(\alpha\otimes 1)(h\otimes m)=j(\alpha(h)\otimes m)=(\alpha(h))(m).\] 
This implies that $\bar{\alpha}\iota(h)=\alpha(h)$ for all $h\in I$. Thus, $\Hom_R(M,A)$ is an injective right $S$-module.

Conversely, let $A_R$ be an $R$-module such that $\Hom_R(M,A)$ is an injective right $S$-module. Let $N$ be an $M$-generated submodule of $M$ and $f:N\to A$ be any $R$-homomorphism. Hence we have the following diagram of right $S$-modules
\[\xymatrix{0 \ar[r] & \Hom_R(M,N)\ar[r]^-{i_\ast} \ar[d]_{f_\ast} & S \ar@{--{>}}[dl]^{\alpha}\\ & \Hom_R(M,A) & }\] 
where $i:N\to M$ is the canonical inclusion. 
By hypothesis there exists an $S$-homomorphism $\alpha:S\to \Hom_R(M,A)$ such that $\alpha i_\ast=f_\ast$. Define $\bar{\alpha}:M\to A$ as $\bar{\alpha}(m)=\left(\alpha(\text{Id}_M)\right)(m)$. 
Let $n\in N$ be arbitrary. Since $N$ is $M$-generated, $n=\sum_{i=1}^kg_i(m_i)$ with $g_i\in\Hom_R(M,N)$ and $m_i\in M$. Then 
\[\bar{\alpha}i(n)=\left(\alpha(\text{Id}_M)\right)(i(n))=\left(\alpha(\text{Id}_M)\right)i\left(\sum_{i=1}^kg_i(m_i)\right)=\sum_{i=1}^k\left(\alpha(\text{Id}_M)\right)ig_i(m_i)=\sum_{i=1}^k\left(\alpha(ig_i)\right)(m_i)\]
\[=\sum_{i=1}^k\left(\alpha i_\ast(g_i)\right)(m_i)=\sum_{i=1}^k\left(f_\ast(g_i)\right)(m_i)=\sum_{i=1}^kfg_i(m_i)=f\left(\sum_{i=1}^kg_i(m_i)\right)=f(n)\]
because $g_i\in S$.
This implies that $f=\bar{\alpha}i$. Thus, $A_R$ is in $\mathfrak{E}_M$.\\
(ii) The proof is similar to that of (i). Note that if $I_S$ is a finitely generated ideal of $S$, then $I\otimes_S M$ is a finitely $M$-generated right $R$-module.
\end{proof}

\begin{rem} (i) For any endoregular module $M$, $\Hom_R(M, A)$ is an f-injective right $S$-module for any module $A$. \\(ii) ({\cite[Theorem 5 and Corollary 2]{megibbenabsolutely}}) Every module over a von Neumann regular ring is f-injective.
\end{rem}

\begin{cor}\label{sremss12} Let $M$ be a right $R$-module with $S=\End_R(M)$ such that $_SM$ is flat. Then the following equivalences hold true:
\begin{enumerate}
\item[(i)]  $M\in \mathfrak{E}_M$ if and only if $S$ is a right self-injective ring.
\item[(ii)]  $M\in\mathfrak{F}_M$ if and only if $S$ is a right f-injective ring.
\end{enumerate}
\end{cor}

Here we have an alternative proof of Theorem \ref{semiendo}.

\begin{cor}
The following conditions are equivalent for a right $R$-module $M$:
\begin{enumerate}
\item[(a)] $M$ is a finite $\Sigma$-Rickart module and $M\in\mathfrak{F}_M$;
\item[(b)] $\End_R(M)$ is a von Neumann regular ring.
\end{enumerate}
\end{cor}

\begin{proof}
The proof follows from Theorems \ref{semiendo} and \ref{endoringsh} and  Corollaries \ref{shabsvn} and \ref{sremss12}(ii). 
\end{proof}

\begin{cor}\label{sremss}
The following conditions are equivalent for a finitely generated module $M$:
\begin{enumerate}
\item[(a)] $M$ is a $\Sigma$-Rickart module and $M\in\mathfrak{E}_M$;
\item[(b)] $\End_R(M)$ is semisimple artinian.
\end{enumerate}
\end{cor}
\begin{proof}
(a)$\Rightarrow$(b) By \cite[Theorem 4.6]{lm1} $S=\End_R(M)$ is a right hereditary ring and $_SM$ is flat. It follows from Corollary \ref{sremss12}(i) that $S$ is right self-injective. Thus, $S_S$ is semisimple artinian by \cite[Corollary]{osofskynonijective}.
(b)$\Rightarrow$(a) Since $S$ is von Neumann regular, from Theorems \ref{semiendo} and \ref{endoringsh} $_SM$ is flat.  So, the proof follows from \cite[Theorem 4.6]{lm1} that $M$ is $\Sigma$-Rickart and from Corollary \ref{sremss12}(i) that $M\in\mathfrak{E}_M$. 
\end{proof}

Continuing the study of the endomorphism ring of a finite $\Sigma$-Rickart module, we study the case when $\End_R(M)$ is a semiprimary ring (Theorem \ref{rickperf}). Recall that a ring $R$ is said to be \emph{semiprimary} if its Jacobson radical, $\Rad R$, is nilpotent and $R/\Rad R$ is a semisimple artinian ring.

%\begin{lem}[{\cite[Lemma 3]{step}}]\label{perricsp}
%Every right perfect right \emph{(}or left\emph{)} Rickart ring is semiprimary, and in particular, left perfect.
%\end{lem}

%The following theorem should be compared with \cite[Theorem 3.4]{rr2}.

%\begin{defn}[{\cite[Definition 4.1]{m}}]
%A module $M$ is said to be \emph{hollow} if every proper submodule of $M$ is small.
%\end{defn}
%
%\begin{rem}
%If a ring $R$ is a local ring with maximum ideal $I$, then $R_R$ is a hollow module because every proper submodule is contained in $I$. Moreover, a hollow module cannot have nonzero proper direct summands, that is, every hollow module is indecomposable.
%\end{rem}

%CHECK AN EXAMPLE. COMPARE EXAMPLE AND THEOREM \ref{rickperf}.
%\begin{exam}
%Let $M$ be $\mathbb{Q}\oplus\mathbb{Q}$ as a $\mathbb{Z}$-module.
%Then $\mathbb{Q}$ is a hollow endoregular $\mathbb{Z}$-module. However, $\End_R(M)=\left(\begin{smallmatrix}\mathbb{Q}&\mathbb{Q}\\\mathbb{Q}&\mathbb{Q}\end{smallmatrix}\right)$ is not a upper triangular matrix.
%\end{exam}

Recall that a ring $R$ is called a \emph{PWD} (\emph{piecewise domain}) if it possesses a complete set $\{e_1,...,e_n\}$ of orthogonal idempotents such that $xy=0$ implies $x=0$ or $y=0$ whenever $x\in e_iRe_k$ and $y\in e_kRe_j$ (see \cite{gs}).

\begin{thm}\label{rickperf}
The following conditions are equivalent for a module $M$:
\begin{enumerate}
\item[(a)] $M$ has a decomposition $\bigoplus_{i=1}^n H_i^{(\ell_i)}$ with $H_i$ an indecomposable endoregular module, $H_i$ is $H_j$-Rickart for all $1\leq i,j\leq n$, and $H_i\ncong H_j$ for $i\neq j$;
\item[(b)] $S=\End_R(M)$ is isomorphic to a upper triangular matrix ring
\[\begin{pmatrix}
\Mat_{\ell_1}(D_1) & V_{12} & V_{13}& \cdots & V_{1n} \\
0 & \Mat_{\ell_2}(D_2) & V_{23}& \cdots & V_{2n} \\
0 & 0 &  \Mat_{\ell_3}(D_3)&\cdots & V_{3n} \\
\vdots & \vdots & \vdots& \ddots&\vdots  \\
0 & 0 & 0 &\cdots & \Mat_{\ell_n}(D_n) 
\end{pmatrix}\]
where $D_i$ is a division ring for $1\leq i\leq n$, and $V_{ij}$ is a $\Mat_{\ell_i}(D_i)$-$\Mat_{\ell_j}(D_j)$-bimodule for all $1\leq i< j\leq n$ satisfying $\mathbf{l}_S(x)\cap V_{ij}=0$ for any $0\neq x\in H_j$. 
%where $M_j=\bigoplus_{k=\ell_{j-1}+1}^{\ell_j}H_k$. 
In particular, $S$ is a semiprimary PWD. 
\end{enumerate}
%Moreover, for a module $M$ satisfying \emph{(a)}, if there exists an ordering $\mathcal{I}=\{1,2,...,n\}$ for the class $\{H_i\}_{i\in I}$ such that $H_i$ is $H_j$-injective for all $i<j$ and $H_i$ is quasi-injective whenever $\ell_i>1$, then $M$ is a Rickart module.
\end{thm}

%\begin{thm}\label{rickperf}
%The following conditions are equivalent for a module $M$:
%\begin{enumerate}
%\item[(a)] $M=\bigoplus_{i=1}^n H_i$ with $H_i$ an indecomposable endoregular module and $H_i$ is $H_j$-Rickart for all $1\leq i,j\leq n$;
%\item[(b)] $S=\End_R(M)$ is isomorphic to a formal matrix ring
%\[\begin{pmatrix}
%\Mat_{\ell_1}(D_1) & V_{12} & \cdots & V_{1m} \\
%0 & \Mat_{\ell_2}(D_2) & \cdots & V_{2m} \\
%\vdots & \vdots & \ddots&\vdots  \\
%0 & 0 & \cdots & \Mat_{\ell_m}(D_m) 
%\end{pmatrix}\]
%where $D_i$ is a division ring for $1\leq i\leq m$, and $V_{ij}$ is an $\Mat_{\ell_i}(D_i)$-$\Mat_{\ell_j}(D_j)$-bimodule for all $1\leq i< j\leq m$ satisfying $\mathbf{l}_S(x)\cap V_{ij}=0$ for any $0\neq x\in H_j$. 
%%where $M_j=\bigoplus_{k=\ell_{j-1}+1}^{\ell_j}H_k$. 
%In particular, $S$ is a semiprimary PWD. 
%\end{enumerate}
%Moreover, for a module $M$ satisfying \emph{(a)}, if there exists an ordering $\mathcal{I}=\{1,2,...,n\}$ for the class $\{H_i\}_{i\in I}$ such that $H_i$ is $H_j$-injective for all $i<j$, then $M$ is a finite $\Sigma$-Rickart module.
%\end{thm}

\begin{proof}
(a)$\Rightarrow$(b) Since each $H_i$ is indecomposable endoregular and $H_i$ is $H_j$-Rickart module, every nonzero homomorphism $\rho:H_i\to H_j$ is a monomorphism. This implies that $S$ is a PWD \cite{gs}. On the other hand, since $H_i\ncong H_j$ for all $1\leq i\neq j\leq n$,  $\Hom_R(H_i,H_j)=0$ or $\Hom_R(H_j,H_i)=0$  from \cite[Proposition 18]{lz1}. Without loss of generality, we can assume that the decomposition $M=\bigoplus_{i=1}^n H_i^{(\ell_i)}$ is such that $\Hom_R(H_i,H_j)=0$ for all $i<j$. Consider the complete set of orthogonal primitive idempotents $\{e_1,\dots,e_n\}$ of $S$ such that $H_i^{(\ell_i)}=e_iM$.  
%Consider the following equivalence relation on $\{1,...,n\}$: $i\sim j$ if and only if $\Hom_R(H_i,H_j)\neq 0$ and $\Hom_R(H_j,H_i)\neq 0$. Note that $i\sim j$ if and only if $H_i\cong H_j$. For, let $\varphi_1:H_i\to H_j$ and $\varphi_2:H_j\to H_i$ be two nonzero homomorphisms. Then $0\neq\varphi_2\varphi_1\in\End_R(H_i)$ and so $\varphi_2\varphi_1$ is an isomorphism because $\End_R(H_i)$ is a division ring from \cite[Proposition 4.4]{lrr5}. This implies that $\varphi_2$ is an epimorphism and hence an isomorphism. Let $C_1,...,C_m$ denote the equivalence classes of this relation and set $e_i=\sum_{j\in C_i}f_j$. Hence $\{e_1,...,e_m\}$ is a complete set of orthogonal idempotents in $S$. 
It follows from \cite[Main Theorem]{gs} that
\[S=\begin{pmatrix}
P_1 & V_{12} & \cdots & V_{1n} \\
0 & P_2 & \cdots & V_{2n} \\
\vdots & \vdots & \ddots&\vdots  \\
0 & 0 & \cdots & P_n 
\end{pmatrix}\]
where $V_{ij}$ is a $P_i$-$P_j$-bimodule and 
$$P_i=\left(\begin{matrix}
D_i & W_{12} & \cdots & W_{1{\ell_i}} \\
W_{21} & D_i & \cdots & W_{2{\ell_i}} \\
\vdots & \vdots & \ddots&\vdots  \\
W_{{\ell_i}1} & W_{{\ell_i}2} & \cdots & D_i
\end{matrix}\right)$$
with $D_i$ a division ring and each $W_{jk}\cong D_i$ as $D_i$-$D_i$-bimodule. That is, $P_i\cong\Mat_{\ell_i}(D_i)$. 
Suppose that $\Hom_R(H_j,H_i)\neq 0$ with $1\leq i< j\leq n$. It follows from \cite[Corollary 2.10]{lrr3} that $H_j$ is $H_i^{(\ell_i)}$-Rickart. Therefore, every nonzero homomorphism in $\Hom_R(H_j,H_i^{(\ell_i)})$ is a monomorphism.
Let $0\neq x\in H_j$ and $\varphi\in S$ such that $\varphi\in\mathbf{l}_S(x)\cap V_{ij}$. 
Then $\varphi(x)=0$. If $V_{ij}=0$, there is nothing to prove. Suppose that  $V_{ij}=\Hom_R(H_j^{(\ell_j)},H_i^{(\ell_i)})\neq 0$.  Assume that $\varphi\neq 0$. Since $\varphi|_{H_j}$ is a monomorphism by the above comment,  $x=0$, a contradiction. Hence $\varphi=0$. Therefore $\mathbf{l}_S(x)\cap V_{ij}=0$.
Note that $\Rad(S)=\left(\begin{smallmatrix}
0 & V_{12} & \cdots & V_{1n} \\\vspace{-0.2cm}
0 & 0 & \cdots & V_{2n} \\
\vdots & \vdots & \ddots&\vdots  \\
0 & 0 & \cdots & 0 
\end{smallmatrix}\right)$ is nilpotent and $S/\Rad(S)\cong \Mat_{\ell_1}(D_1)\times\cdots \times \Mat_{\ell_n}(D_n)$ which is a semisimple artinian ring. Hence $S$ is a semiprimary ring.

(b)$\Rightarrow$(a) Let $\{e_{ij}\mid 1\leq i,j\leq m\}$ denote the matrix units where $m=\ell_1+\cdots+\ell_n$. Hence $M$ has a decomposition $M=e_{11}M\oplus\cdots\oplus e_{mm}M$. Denote $H_i=e_{ii}M$. Since $\End_R(H_i)\cong e_{ii}Se_{ii}\cong D_i$ is a division ring, $H_i$ is an indecomposable endoregular $R$-module for all $1\leq i\leq m$. 
By hypothesis, $H_j\cong H_k$ for $m_{i-1}< j, k \leq m_i$ where $m_i=\sum_{k=0}^i \ell_k$ and $\ell_0=0$ and for each $1\leq i\leq n$. Hence $M=M_1\oplus\cdots \oplus M_n$ where $M_i=\bigoplus_{k=m_{i-1}+1}^{m_i}H_k\cong H_i^{(\ell_i)}$. Without loss of generality we can take a summand of each $M_i$ and assume that $M=H_1^{(\ell_1)}\oplus\cdots\oplus H_n^{(\ell_n)}$. Let $\rho:H_j\to H_i$ be any nonzero homomorphism for $1\leq i< j\leq n$. Assume that $0\neq x\in H_j$ such that $\rho(x)=0$. Consider $\rho\oplus 0:H_j^{(\ell_j)}\to H_i^{(\ell_i)}$. Then $(\rho\oplus 0)(x)=0$. This implies that $(\rho\oplus 0)\in\mathbf{l}_S(x)\cap V_{ij}=0$, a contradiction. Therefore $x=0$. This 
 implies that $\rho$ is a monomorphism. Hence it is easy to see that $H_i$ is $H_j$-Rickart for all $1\leq i, j\leq n$.
%, it follows that $\Hom_R(M_j,M_i)=0$ or each nonzero homomorphism in $\Hom_R(M_j,M_i)$ is a monomorphism. This implies that $H_j$ is $H_i$-Rickart for all $1\leq i,j\leq n$.
%Moreover, suppose that $M$ satisfies (a) and there exists an ordering $\mathcal{I}=\{1,2,...,n\}$ for the class $\{H_i\}_{i\in \mathcal{I}}$ such that $H_i$ is $H_j$-injective for all $i<j$ and that $H_i$ is quasi-injective if $\ell_i>1$.  It follows from \cite[Corollary 2.10 and Theorem 2.12]{lrr3} that $H_i^{(\ell_i)}$ is $H_j^{(\ell_j)}$-Rickart. On the other hand, $H_i^{(\ell_i)}$ is $H_j^{(\ell_j)}$-injective for all $i<j$ by \cite[16.1 and 16.2]{wisbauerfoundations}. Therefore $M=H_1^{(\ell_1)}\oplus \cdots\oplus H_n^{(\ell_n)}$ is a Rickart module by \cite[Corollary 2.13]{lrr3}.
\end{proof}

\begin{rem}
If $M$ is a finite direct sum of indecomposable endoregular modules, then $\text{End}_R(M)$ is a semi-perfect ring.
\end{rem}

\begin{cor}\label{corrickperf}
Suppose that $M=\bigoplus_{i=1}^n H_i$ with $H_i$ an indecomposable endoregular module, $H_i$ is $H_j$-Rickart for all $1\leq i,j\leq n$ and $H_i\ncong H_j$ for $i\neq j$. If there exists an ordering $\mathcal{I}_n=\{1,2,...,n\}$ for the class $\{H_i\}_{i\in \mathcal{I}_n}$ such that $H_i$ is $H_j$-injective for all $i<j$, then $M$ is a Rickart module and $\End_R(M)$ is isomorphic to a upper triangular  matrix ring
\[\begin{pmatrix}
D_1 & V_{12}& V_{13} & \cdots & V_{1n} \\
0 & D_2 & V_{23} & \cdots & V_{2n} \\
0 & 0  & D_3& \cdots & V_{3n} \\
\vdots & \vdots & \vdots & \ddots&\vdots  \\
0 & 0 & 0& \cdots & D_n 
\end{pmatrix}\]
where $D_i$ is a division ring for $1\leq i\leq n$, and $V_{ij}$ is a $D_i$-$D_j$-bimodule for all $1\leq i< j\leq n$. In particular, $\End_R(M)$ is a semiprimary PWD.
\end{cor}

\begin{proof}
It follows directly from \cite[Corollary 2.13]{lrr3} and Theorem \ref{rickperf}. 
\end{proof}

The next example illustrates Theorem \ref{rickperf} and Corollary \ref{corrickperf}.

\begin{exam}\label{exthcor}
Let $F$ be a field and $R=\left( \begin{smallmatrix}
F & F \\
0 & F
\end{smallmatrix}\right)$. Consider the right $R$-module 
\[M=\left( \begin{smallmatrix}
F & F \\
0 & 0
\end{smallmatrix}\right) \oplus \left( \begin{smallmatrix}
0 & 0 \\
0 & F
\end{smallmatrix}\right) \oplus \left( \begin{smallmatrix}
0 & 0 \\
0 & F
\end{smallmatrix}\right)=\left( \begin{smallmatrix}
F & F \\
0 & 0
\end{smallmatrix}\right) \oplus \left( \begin{smallmatrix}
0 & 0 \\
0 & F
\end{smallmatrix}\right)^{(2)}.\] 
Denote $H_1=\left( \begin{smallmatrix}
F & F \\
0 & 0
\end{smallmatrix}\right)$ and $H_2=\left( \begin{smallmatrix}
0 & 0 \\
0 & F
\end{smallmatrix}\right)$. Then $\End_R(H_1)=F=\End_R(H_2)$, $\Hom_R(H_2,H_1)=F$ and $\Hom_R(H_1,H_2)=0$. Thus, $M$ satisfies the condition (a) of Theorem \ref{rickperf}. Hence the endomorphism ring of $M$ is 
\[\End_R(M)\cong\left(\begin{smallmatrix}
F & F & F \\
0 & F & F \\
0 & F & F
\end{smallmatrix} \right).\]
Moreover, $H_1$ is $H_2^{(2)}$-injective because $H_2$ is simple, therefore $M=H_1\oplus H_2^{(2)}$ is a Rickart module. In particular, $R$ is a right Rickart ring.
\end{exam} 

Inspired by the last example, we have the following proposition.

\begin{prop}\label{mandsimple}
Let $M$ be a finite $\Sigma$-Rickart module and $P$ be any simple module such that $\Hom_{R}(M,P)=0$. Then $M^{(\ell)}\oplus P^{(n)}$ is a finite $\Sigma$-Rickart module for any $\ell, n>0$.
\end{prop}

\begin{proof}
Let $k>0$. Hence $(M^{(\ell)}\oplus P^{(n)})^{(k)}=M^{(k\ell)}\oplus P^{(kn)}$. It is clear that $P^{(kn)}$ is $M^{(k\ell)}$-Rickart and by hypothesis, $M^{(k\ell)}$ is $P^{(kn)}$-Rickart. Since $P^{(kn)}$ is semisimple, $M^{(k\ell)}$ is $P^{(kn)}$-injective. It follows from \cite[Corollary 2.13]{lrr3} that $M^{(k\ell)}\oplus P^{(kn)}$ is a Rickart module. Thus, $M^{(\ell)}\oplus P^{(n)}$ is a finite $\Sigma$-Rickart module for any $\ell,n>0$.
\end{proof}

\begin{rem}
It follows from Proposition \ref{mandsimple} that the module $M$ in Example \ref{exthcor} is not just Rickart, but finite $\Sigma$-Rickart. In particular, the ring $R=\left( \begin{smallmatrix}
F & F \\
0 & F
\end{smallmatrix}\right)$ is right hereditary.
\end{rem}

\begin{cor}\label{ringmatrix}
The following conditions are equivalent for a ring $R$:
\begin{enumerate}
\item[(a)] $R=\bigoplus_{i=1}^n I_i^{(\ell_i)}$ with $I_i$ an endoregular right ideal and $I_i$ is $I_j$-Rickart for all $1\leq i,j\leq n$;
\item[(b)] $R$ is isomorphic to a formal matrix ring
\[\begin{pmatrix}
\Mat_{\ell_1}(D_1) & V_{12} & \cdots & V_{1n} \\
0 & \Mat_{\ell_2}(D_2) & \cdots & V_{2n} \\
\vdots & \vdots & \ddots&\vdots \\
0 & 0 & \cdots & \Mat_{\ell_n}(D_n) 
\end{pmatrix}\]
where $D_i$ is a division ring for $1\leq i\leq n$, and for all $1\leq i< j\leq n$, $V_{ij}$ is a $\Mat_{\ell_i}(D_i)$-$\Mat_{\ell_j}(D_j)$-bimodule satisfying $\mathbf{l}_R(x)\cap V_{ij}=0$ for any $0\neq x\in I_j$. In particular, $R$ is a semiprimary PWD. 
\end{enumerate}
%Moreover, if there exists an ordering $\mathcal{I}=\{1,2,...,n\}$ for the class $\{I_i\}_{i\in \mathcal{I}}$ such that $I_i$ is $I_j$-injective for all $i<j$, then $R$ is a (semi-)hereditary ring.
\end{cor}

%\begin{cor}\label{ringsemiprim}
%The following conditions are equivalent for a right Rickart ring $R$:
%\begin{enumerate}
%\item[(a)] $R$ is a semiprimary ring; 
%\item[(b)] $R=\bigoplus_{i=1}^n I_i$ with $I_i$ a hollow endoregular right ideal;
%\item[(c)] There exists a complete set of orthogonal primitive idempotents $\{e_1,\dots,e_n\}$ of $R$ such that $e_iR$ is hollow endoregular.
%\end{enumerate}
%\end{cor}

To illustrate the last results, we have the following examples:

\begin{exam}
(i) Let $K$ be a field. It follows from \cite[Ch.I Lemma 1.12 and Ch. VII Theorem 1.7]{repth} that every path algebra $K\mathsf{Q}$ of a finite, connected and acyclic quiver $\mathsf{Q}$ satisfies the conditions of Corollary \ref{ringmatrix}.

(ii) Let $K$ and $F$ be division rings and $U$ be any left $K$- right $F$-bimodule. Then the formal matrix ring $R=\left(\begin{smallmatrix}
K & U \\
0 & F
\end{smallmatrix} \right)$ trivially satisfies the condition (b) of Corollary \ref{ringmatrix}. Moreover, the decomposition $R=\left(\begin{smallmatrix}
K & U \\
0 & 0
\end{smallmatrix} \right)\oplus \left(\begin{smallmatrix}
0 & 0 \\
0 & F
\end{smallmatrix} \right)$ makes $R$ to be a hereditary semiprimary PWD by Corollary \ref{ringmatrix} and Proposition \ref{mandsimple}. 

(iii) Let $K$ be a field and consider the ring $R=\left( \begin{smallmatrix}
K & 0 & 0 \\
K & K & K \\
0 & 0 & K
\end{smallmatrix}\right)$. Then $R$ has a decomposition in hollow endoregular right ideals
\[R=\left(\begin{smallmatrix}
0 & 0 & 0 \\
K & K & K \\
0 & 0 & 0
\end{smallmatrix}\right)\oplus\left(\begin{smallmatrix}
K & 0 & 0 \\
0 & 0 & 0 \\
0 & 0 & 0
\end{smallmatrix} \right)\oplus\left(\begin{smallmatrix}
0 & 0 & 0 \\
0 & 0 & 0 \\
0 & 0 & K
\end{smallmatrix}\right).\]
Denote those summands by $I_1,I_2$ and $I_3$ from the left to the right respectively. Hence $I_2$ and $I_3$ are simple $R$-modules. By Corollary \ref{ringmatrix} and applying Proposition \ref{mandsimple} twice, we have that $R$ is a (semi-)hereditary semiprimary PWD.
%This implies that $I_1$ is $I_2$-injective and $I_3$-injective, and $I_2$ is $I_3$-injective. Thus, 
%Note that $\Rad(R)=\left(\begin{smallmatrix}
%0 & 0 & 0 \\
%K & 0 & K \\
%0 & 0 & 0
%\end{smallmatrix} \right)$ and $\Rad(R)^2=0$. 
Moreover by Corollary \ref{ringmatrix}(b), $R\cong\left(\begin{smallmatrix}
K & K & K \\
0 & K & 0 \\
0 & 0 & K
\end{smallmatrix} \right)$ where the isomorphism is given by $\left(\begin{smallmatrix}
a & 0 & 0 \\
b & c & d \\
0 & 0 & e
\end{smallmatrix} \right)\mapsto\left(\begin{smallmatrix}
c & b & d \\
0 & a & 0 \\
0 & 0 & e
\end{smallmatrix} \right)$.

(iv) Let $K$ be a field and $K[x]$ be the polynomial ring with coefficients in $K$. Consider the ring $R=\left(\begin{smallmatrix}
K & \left\langle x \right\rangle & 0 \\
0 & K & K[x]/\left\langle x \right\rangle \\
0 & 0 & K
\end{smallmatrix}\right)$ where $\left\langle x \right\rangle$ is the ideal generated by $x$. Hence $R$ has a decomposition in indecomposable endoregular ideals
\[R=\left(\begin{smallmatrix}
K & \left\langle x \right\rangle & 0 \\
0 & 0 & 0 \\
0 & 0 & 0
\end{smallmatrix}\right)\oplus\left(\begin{smallmatrix}
0 & 0 & 0 \\
0 & K & K[x]/\left\langle x \right\rangle \\
0 & 0 & 0
\end{smallmatrix}\right)\oplus\left(\begin{smallmatrix}
0 & 0 & 0 \\
0 & 0 & 0 \\
0 & 0 & K
\end{smallmatrix}\right).\]
Let $I_1,I_2,I_3$ denote those summand from the left to the right respectively. Let $0\neq \overline{f(x)}\in K[x]/\left\langle x \right\rangle$ then $\left(\begin{smallmatrix}
0 & x & 0\\
0 & 0 & 0\\
0 & 0 & 0
\end{smallmatrix}\right)\in\mathbf{l}_R\left(\begin{smallmatrix}
0 & 0 & 0\\
0 & 0 & \overline{f(x)} \\
0 & 0 & 0
\end{smallmatrix} \right)\cap I_1$. Thus, $R$ does not satisfies the condition (b) in Corollary \ref{ringmatrix}. Note that $R$ is a semiprimary ring.
\end{exam}

%%%%%%%%%%%%%%%%%%%%%%%%%%%%%%%%%%%%%%%%%%%%%%%%%%%%%%%%%%%%%
%%%%%%%%%%%%%%%%%%%%%%%%%%%%%%%%%%%%%%%%%%%%%%%%%%%%%%%%%%%%%%
%여기서부터 줄임.
%\begin{lem}[{\cite[43.8]{wisbauerfoundations}}]\label{D1perf2}
%The following conditions are equivalent for a finitely generated quasi-projective module $M$:
%\begin{enumerate}
%\item[(a)] $M^{(\mathbb{N})}$ is semiperfect in $\sigma[M]$.
%\item[(b)] $\End_R(M)$ is a right perfect ring.
%\item[(c)] $M/\Rad M$ is semisimple and $\Rad M^{(\mathbb{N})}\ll M^{(\mathbb{N})}$.
%\end{enumerate}
%\end{lem} 

%In order to get a better understanding of the next results and avoid too many references in their proofs, we resume some results in the following lemmas.
%
%\begin{lem}\label{D1perf}
%Let $M$ be a finitely generated quasi-projective module. Consider the following conditions:
%\begin{enumerate}
%\item[(i)] $M^{(\mathbb{N})}$ has $D_1$ condition.
%\item[(ii)] $\End_R(M)$ is a right perfect ring.
%\item[(iii)] $M$ has $D_1$ condition. 
%\end{enumerate}
%Then the implications $(\emph{i})\Rightarrow(\emph{ii})\Rightarrow(\emph{iii})$ hold true.
%\end{lem} 

\begin{defn}[{\cite[Definition 2.23]{m}}]
A family of modules $\{M_\alpha\mid\alpha\in\Lambda\}$ is said to be \emph{locally-semi-Transfinitely-nilpotent} (\emph{lsTn}) if for any subfamily $M_{\alpha_i}$ $(i\in\mathbb{N})$ with distinct $\alpha_i$ and any family of non-isomorphisms $\varphi_i:M_{\alpha_i}\to M_{\alpha_{i+1}}$, and for every $x\in M_{\alpha_1}$, there exists $n\in\mathbb{N}$ (depending on $x$) such that $\varphi_n\cdots\varphi_2\varphi_1(x)=0$.
\end{defn}

\begin{prop}\label{ricksp}
Consider the following conditions for a Rickart module $M$:
\begin{enumerate}
\item[(i)] $M=\bigoplus_{i=1}^n H_i$ with $H_i$ a hollow endoregular module;
\item[(ii)] $\End_R(M)$ is a semiprimary ring.
\end{enumerate}
Then \emph{(i)}$\Rightarrow$\emph{(ii)}. In addition, if $M$ is finitely generated then the two conditions are equivalent.
\end{prop}

\begin{proof}
(i)$\Rightarrow$(ii) Since $M$ is a Rickart module, $H_i$ is $H_j$-Rickart. It follows from Theorem \ref{rickperf} that $S$ is semiprimary.
%Since $M$ is finitely generated quasi-projective, $H_i$ is $H_j^{(\mathcal{I})}$-projective for every index set $\mathcal{I}$. 
%On the other hand, any $0\neq\varphi:H_i\to H_j$ is a monomorphism because $H_i$ is $H_j$-Rickart and $H_i$ is indecomposable (see \cite[Proposition 4.4]{lrr5}) and also $\End_R(H_i)$ is a division ring for all $1\leq i\leq n$ because $H_i$ is indecomposable endoregular. Let $\mathcal{I}$ be any index set and consider the family
%\begin{equation}\label{family}
%\{H_{1_j}\mid H_1=H_{1_j},\;j\in\mathcal{I}\}\cup\cdots\cup\{H_{n_j}\mid H_n=H_{n_j},\; j\in\mathcal{I}\}.
%\end{equation}
%Then for any subfamily $\{H_{j_\alpha}\mid \alpha\in\mathbb{N}\}$ and homomorphisms $\varphi_{\alpha}:H_{j_\alpha}\to H_{j_{\alpha+1}}$ there must be an homomorphism $\varphi_\beta=0$ because the family $\{H_{j_\alpha}\mid \alpha\in\mathbb{N}\}$ has more than $n$ distinct elements. Hence the family (\ref{family}) satisfies lsTn. Thus $M^{(\mathcal{I})}=\bigoplus_{i=1}^n H_i^{(\mathcal{I})}$ is a quasi-discrete module by \cite[Corollary 4.49]{m}. In particular, $M^{(\mathbb{N})}$ has $D_1$ condition. Thus $S$ is a right perfect ring by Lemma \ref{D1perf}. Therefore Lemma \ref{perricsp} yields $S$ is a semiprimary ring.

In addition, suppose $M$ is finitely generated.
(ii)$\Rightarrow$(i) Suppose $S=\End_R(M)$ is a semiprimary ring. Then $S$ is right perfect. From Lemma \cite[43.8]{wisbauerfoundations} $M$ has $D_1$ condition. 
Also $M$ is quasi-discrete because $M$ is Rickart.
Therefore $M$ has an irredundant decomposition $M=\bigoplus_{i=1}^n H_i$ with $H_i$ a hollow module such that complements summands and is unique up to isomorphism by \cite[Theorem 4.15]{m}. Since
\[M^{(\mathbb{N})}=\bigoplus_{i=1}^n H_i^{(\mathbb{N})}\]
and by \cite[43.8]{wisbauerfoundations}, $\Rad M^{(\mathbb{N})}\ll M^{(\mathbb{N})}$, the module $M^{(\mathbb{N})}$ is quasi-discrete and the family 
\[\{H_1=H_{1_j}\mid j\in\mathbb{N}\}\cup\cdots\cup\{H_n=H_{n_j}\mid j\in\mathbb{N}\}\]
satisfies lsTn by \cite[Theorem 4.53 and Corollary 4.49]{m}. Since $H_i$ is indecomposable Rickart for all $1\leq i\leq n$, every nonzero endomorphism $\varphi:H_i\to H_i$ is a monomorphism. By lsTn, $\varphi$ must be an isomorphism. Thus $H_i$ is an endoregular module for all $1\leq i\leq n$.
\end{proof}

With the following examples we will show that the hypothesis on $M$ to be Rickart and on $H_i$ to be endoregular in Proposition \ref{ricksp} are not superfluous.

\begin{exam}
(i) Set $M=\mathbb{Z}_4$ as $\mathbb{Z}$-module. It is clear that $\End_\mathbb{Z}(M)=\mathbb{Z}_4$ is a semiprimary ring and $M$ is a hollow module. But $M$ is neither Rickart nor endoregular.

(ii) Let $R=\mathbb{Z}_{(p)}$ be the localization of integers at a prime $p$ and set $M=R_R$. Then $M$ is a finitely generated Rickart module. Note that $M$ is a hollow $R$-module but is not endoregular. Moreover $R=\End_R(M)$ is not a semiprimary ring because $\text{Rad}(\mathbb{Z}_{(p)})=p\mathbb{Z}_{(p)}$.
\end{exam}

It follows from an Auslander's result \cite[5.72]{l} that a semiprimary right semi-hereditary ring is right and left hereditary. The corollaries below give an extension of Auslander's result for the case of $\Sigma$-Rickart and fintie $\Sigma$-Rickart modules.

\begin{cor}\label{rickperf1}
Let $M$ be a finite $\Sigma$-Rickart module. If $M=\bigoplus_{i=1}^n H_i$ with $H_i$ a hollow endoregular module then $\End_R(M^{(\ell)})\cong\Mat_\ell(S)$ is a semiprimary \emph{(}right\emph{)} hereditary PWD for all $\ell>0$ where $S=\End_R(M)$.
\end{cor}

\begin{proof}
Let $\ell>0$. Then $M^{(\ell)}=\bigoplus_{i=1}^n H_i^{(\ell)}$ with each $H_i$ hollow endoregular. Since $M$ is finite $\Sigma$-Rickart and finitely generated, $M^{(\ell)}$ is a Rickart module. From Theorem \ref{rickperf} $\End_R(M^{(\ell)})$ is a semiprimary PWD. On the other hand, $\End_R(M^{(\ell)})=\Mat_\ell(S)$ is a semi-hereditary ring by Theorem \ref{endoringsh}.
From \cite[5.72]{l} $\End_R(M^{(\ell)})$ is a semiprimary hereditary PWD.
\end{proof}

\begin{cor}\label{fsrissr}
Let $M$ be a finitely generated module such that $M=\bigoplus_{i=1}^n H_i$ with $H_i$ a hollow endoregular module. The following conditions are equivalent:
\begin{enumerate}
\item[(a)] $M$ is a $\Sigma$-Rickart module;
\item[(b)] $M$ is a finite $\Sigma$-Rickart module.
%\item[(c)] every $M$-generated submodule of any element in $\Add(M)$ has $D_2$ condition;
%\item[(d)] every $M$-generated submodule of any element in $\Add(M)$ is quasi-projective;
%\item[(e)] every finitely $M$-generated submodule of any element of $\add(M)$ is $M$-projective;
%\item[(f)] every finitely $M$-generated submodule of any element of $\add(M)$ is quasi-projective.
\end{enumerate} 
\end{cor}

\begin{proof}
(a)$\Rightarrow$(b) is clear. For, (b)$\Rightarrow$(a) $\End_R(M)$ is a semiprimary hereditary ring, by Corollary \ref{rickperf1}. It follows from \cite[Theorem 4.6]{lm1} that $M$ is a $\Sigma$-Rickart module. 
%(a)$\Leftrightarrow$(c)$\Leftrightarrow$(d) follow from \cite[Theorem 2.17]{lm1} and (b)$\Leftrightarrow$(e)$\Leftrightarrow$(f) follows from Corollary \ref{5}.
\end{proof}

\begin{cor}\label{srfsr}
Consider the following conditions for a module $M$:
\begin{enumerate}
\item[(i)] $M$ is \emph{(}finite\emph{)} $\Sigma$-Rickart and $M=\bigoplus_{i=1}^n H_i$ with $H_i$ a hollow endoregular module;
\item[(ii)] $S=\End_R(K)$ is a semiprimary \emph{(}right\emph{)} hereditary ring for every $K$ in $\add(M)$ and $_SK$ is flat.
\end{enumerate}
Then \emph{(i)}$\Rightarrow$\emph{(ii)}. In addition, if $M$ is finitely generated, then the two conditions are equivalent.
\end{cor}

\begin{proof}
(i)$\Rightarrow$(ii) Let $K\in \add(M)$. Then $M^{(n)}=K\oplus L$ for some $n>0$. Hence $K$ is a finite $\Sigma$-Rickart module. On the other hand, $M^{(n)}$ has the cancellation property, by \cite[Corollary 4.20]{m}. Therefore $K$ satisfies the hypothesis of Corollary \ref{rickperf1}.  Thus $\End_R(K)$ is a semiprimary (right) hereditary ring. (ii)$\Rightarrow$(i) follows from \cite[Theorem 4.6]{lm1} and Proposition \ref{ricksp}.
\end{proof}

\centerline{Acknowledgments}
\vspace{0.2cm}

%The authors are grateful to Professor Jae Keol Park for his helpful suggestions and also express their sincere gratitude to the referee for a prompt and thorough report.
The authors are very thankful to Research Institute of Mathematical Sciences, Chungnam National University (CNU-RIMS), Republic of Korea, for the support of this research work. The first author gratefully acknowledges the support of this research work by the National Research Foundation of Korea (NRF) grant funded by the Korea government(MSIT)(2019R1F1A105988312)

\end{document}